\documentclass[12pt]{amsart}%
\usepackage[margin=1.2in]{geometry}
\usepackage{amsfonts}
\usepackage{amsmath}
\usepackage{amssymb}
\usepackage{graphicx}%

\usepackage[T1]{fontenc}
\usepackage{tgtermes}  
\usepackage{newtxmath} 

\usepackage{multirow}
\usepackage{array}

\usepackage[colorlinks, linkcolor=blue,  citecolor=blue, urlcolor=blue, bookmarks=false]{hyperref}%
\hypersetup{pdfstartview=FitH}
\usepackage{url}
\usepackage{xspace}

\usepackage{dsfont}

\usepackage{verbatim}
\usepackage{alltt}
\usepackage{graphicx}
\graphicspath{{../figures/}{../figuresOld/}}

\newtheorem{theorem}{Theorem}

\newtheorem{corollary}{Corollary}

\newtheorem{definition}{Definition}
\newtheorem{example}{Example}

\newtheorem{lemma}{Lemma}

\newtheorem{remark}{Remark}

\numberwithin{equation}{section}

\def\XXint#1#2#3{{\setbox0=\hbox{$#1{#2#3}{\int}$}
     \vcenter{\hbox{$#2#3$}}\kern-.5\wd0}}

\let\originalleft\left
\let\originalright\right
\renewcommand{\left}{\mathopen{}\mathclose\bgroup\originalleft}
\renewcommand{\right}{\aftergroup\egroup\originalright}

\usepackage{mathrsfs}
\renewcommand{\mathrm}{\mathscr} 
\renewcommand{\mathcal}{\mathscr}

\renewcommand{\ref}[1]{\autoref{#1}}

\begin{document}
\title[Central limit theorem for Wasserstein projection]{Central limit theorem for Wasserstein projection\\[0.5em]the case of convex order}
\author{Yuanlong Ruan}
\address{Beihang University, Beijing, China}

\date{\today}

\begin{abstract}
The main focuses of the article are limit theorems of Wasserstein projection in the convex order which are useful for inference tasks. The main results rely on the establishment of dual attainment, stability and several useful observations. The first is a clean criterion to invoke the Wasserstein projection dualities on the classical cost $c=h(x-y)$. A new strategy was introduced to establish the dual attainment. Backward and forward dual transport are proved to enjoy a conjugate relationship. Sufficient conditions are given that ensure uniqueness of the optimal dual potential. These ingredients laid the groundwork for the proof of optimal dual potential stability, thereby permitting us to prove the central limit theorems under mild moment assumptions. Additional discussions show that the moment assumptions are sharp. These results also answered several open questions raised by Professor Benjamin Jourdain. 

\end{abstract}
\maketitle

\section{Introduction and main results}

Convex ordering has wide applications in probability and statistics. It helps
compare random variables according to risk spreads over convex functions. But
the comparing requires evaluation of expected values across an infinite number
of convex functions, which is practically impossible. By contrast, Wasserstein
projection in the convex order packs these evaluations into a single quantity,
rendering it suited for a number of quantitative purposes. With respect to a
cost $c\left(  x,y\right)  ,$ the backward and forward Wasserstein projections
are defined respectively as%
\[
T_{c}\left(  \mu,P_{\leqslant\nu}\right)  =\inf_{\eta\in P_{\leqslant\nu}%
}T_{c}\left(  \mu,\eta\right)  ,\text{ }T_{c}\left(  P_{\mu\leqslant}%
,\nu\right)  =\inf_{\xi\in P_{\mu\leqslant}}T_{c}\left(  \xi,\nu\right)  .
\]
where $T_{c}\left(  \mu,\eta\right)  $ is the optimal transport between $\mu,$
$\eta$ with cost function $c$, $P_{\leqslant\nu}$ is the cone of probabilities
dominated by $\nu$ in the convex order$,$ and $P_{\mu\leqslant}$ the cone of
those dominating $\mu.$

Wasserstein projection in the convex order is initially introduced as an
efficient device \cite{alfonsi2020sampling} to address the difficulty in
sampling from martingale couplings. It is also investigated as a weak optimal
transport with barycentric cost \cite{gozlan2020mixture}%
\cite{gozlan2017kantorovich}. Dualities for Wasserstein projections are proved
in \cite{yh_yl_stochastic_order}. Recent interests have extended to projection
stability via metric extrapolation \cite{kim2025stability}, regularity and
characterization in the quadratic Gaussian case \cite{alfonsi2026wasserstein}
and reconstruction of Laguerre tessellation \cite{bourne2026semi}. The
defining feature of Wasserstein projection in the convex order has also made
it ideal for related statistical pusposes \cite{kim2024statistical}.

In statistical optimal transport, central limit theorem is at the heart of
inference tasks. A great deal of recent effort has sought to establish limit
theorems, ranging from the seminal work on the classical theme
\cite{del2019central}\cite{del2024central}, semi-discrete Wasserstein
distances \cite{del2024central} to entropic transportation costs
\cite{del2023improved}, and more recently entropic potential plans, maps and
divergences \cite{gonzalez2022weak}, \cite{goldfeld2024limit}. For a
comprehensive review, we refer to \cite{del2025distributional}.

In comparison, the statistical study of Wasserstein projection is still in the
early stages. This article intends to provide limit theorems for Wasserstein
projections in the convex order that will be helpful for inference tasks such
as \cite{kim2024statistical}. The final limit theorems are built upon several
important ingredients. We first give a sufficient and necessary condition
(Lemma \ref{lm:cost_equiv}) that allows us to invoke the duality theorems on a
cost of the classical form $c\left(  x,y\right)  =h\left(  x-y\right)  .$ We
established the backward-forward swap property in Theorem \ref{thm:bf_swap},
this property together with a new strategy is used to tackle the dual
attainment for both backward and forward projection in Theorem
\ref{thm:dual_att}. Projection and attainment under Lipschitz cost $c\left(
x,y\right)  =\left\vert x-y\right\vert $ are given. We note that previously
the only attainment available is done for quadraric cost.

The attainment together with the backward-forward swap property also leads to
a conjugate relation (Theorem \ref{thm:conj_opt_pair}) between a backward dual
optimizer $\varphi$ and a forward dual optimizer $\psi$%
\[
\varphi=Q_{\bar{c}}\left(  \psi\right)  ,\text{ }\psi=Q_{c}\left(
\varphi\right)
\]
which holds everywhere. This shows that the backward and forward projection
share the same optimal mapping (if it exists), although backward and forward
projection are conceptually different. To be precise, if $\bar{\mu}$,
$\bar{\nu}$ are respectively optimal backward and forward projection
measure$,$ then the conjugate relation tells us that the transport between
$\mu,$ $\bar{\mu}$ and between $\bar{\nu}$, $\nu$ are supplied by the same
optimal dual potential pair $\left(  Q_{c}\left(  \varphi\right)
,\varphi\right)  =\left(  \psi,Q_{\bar{c}}\left(  \psi\right)  \right)  $.

A further ingredient towards central limit theorems is the uniqueness of dual
potentials. In Lemma \ref{lm:uniq1} and Lemma \ref{lm:uniq2}, we provide two
complementary sufficient conditions, one of them requires absolute continuity
of one of the marginals, while the other does not.

Stability stands at the center of this article. Our approach identifies a
convex set containing the support of the the pre-perturbed marginals on which
the optimal dual potentials are locally uniformly convergent. Note that this
is stronger than convergence that only takes place in the interior of the
supports. Under mild conditions together with the aid of a special half-space
property of the perturbed optimal potentials in (Lemma \ref{lm:half_spc}), we
give the stability result in Theorem \ref{thm:stb_potential}.

Central limit theorems are established for backward and forward projection. As
with\ \cite{del2019central}\cite{del2024central}, we build the proof on top of
concentration inequality. It is worth noting that the proof departs from the
case of optimal transport in a crucial step where a variance bound is derived.
The difficulty is structural. In optimal transport between $\mu$ and $\nu,$
assume $\mu$ is approximated by empirical random measures $\mu_{m}.$ To derive
the desired variance bound, the transportation from $\mu_{m}$ is partitioned
according to the inverse map from $\nu$ to $\mu_{m}.$ For this to work $\nu$
has to be absolutely continuous, or approximated by absolutely continuous
measures. This inverse map partitioning approach fail to work for measures in
convex order. We take a different approach via the Lipschitz stability of the
projection distances. When perturbation happens at the non-vertex side of the
backward projection $T_{c}\left(  \mu,P_{\leqslant\nu}\right)  :\mu$ is
approximated by its empirical versions $\mu_{m},$ the limit theorem is clean
and no higher order moment is required (Theorem \ref{thm:CLT}),%
\[
\sqrt{m}\left[  T_{p}\left(  \mu_{m},P_{\leqslant\nu}\right)  -E\left(
T_{p}\left(  \mu_{m},P_{\leqslant\nu}\right)  \right)  \right]  \rightarrow
\mathcal{N}\left(  0,\sigma_{p,\mu}^{2}\right)  \text{ weakly},
\]
where $\sigma_{p,\mu}^{2}$ is given in $\left(  \ref{thm:CLT_eq2}\right)  .$
By contrast, when perturbation happens at the vertex side$:\nu$ is
approximated by its empirical versions $\nu_{n},$ higher order moments are
required to obtain variance convergence and the limit theorem (Theorem
\ref{thm:CLT_vtx}). The limit theorems are proved for $p$-cost $\left\vert
x-y\right\vert ^{p}$ $\left(  p>1\right)  .$ For general strictly convex cost,
the results still hold, although some Orlicz space language will be needed. In
the end we show that in general the moment assumption in the limit theorems
are sharp. One exception is the special case where both $\mu,$ $\nu$ have
bounded support.

The results in this article also answered several open questions raised by
Professor Benjamin Jourdain in a Fields Institute talk
\cite{fields2025Benjamin}: does the dual attainment hold for costs other than
the quadratic, whether the nice properties for quadratic cost extends to other
costs that backward and forward projection share the same transport map.

Finally, in light of the relation between backward Wasserstein projection and
the weak optimal transport with barycentric cost, the central limit theorems
presented here readily yield the central limit theorems for the weak optimal
transport with barycentric cost.

\section{Notations and preliminaries}

$\operatorname*{int}\left(  U\right)  :$ the interior of a set $U.$

$\operatorname*{dom}\left(  f\right)  :$ $\left\{  x\in\mathbb{R}^{d}:f\left(
x\right)  <\infty\right\}  .$

$\operatorname*{supp}\left(  \mu\right)  ,$ $\operatorname*{int}\left(
\operatorname*{supp}\left(  \mu\right)  \right)  :$ the support of $\mu$ and
its interior.

$C_{b,p}:$ bounded continuous function modulo the factor $1+\left\vert
x\right\vert ^{p}$ (see \cite[Section 4.1]{yh_yl_stochastic_order}).

$P\left(  \mathbb{R}^{d}\right)  $ (resp. $P_{p}\left(  \mathbb{R}^{d}\right)
$) for the set of probabilities (of $p$-th moment)

$P_{p}^{ac}\left(  \mathbb{R}^{d}\right)  :$ the set of probabilities of
$p$-th moment that are absolutely continuous (w.r.t. the Lebesgue measure).

\begin{definition}
\label{def:the_class}Let $\mu,$ $\nu\in P\left(  \mathbb{R}^{d}\right)  $ and%
\[
\mathcal{A=}\text{ the set of real-valued convex functions on }\mathbb{R}%
^{d}.
\]
We say that $\mu$ is smaller than $\nu$ in the convex order (or equivalently
$\nu$ is greater than $\mu$ in the convex order), written $\mu\leqslant
_{cx}\nu$, if%
\[
\int\varphi d\mu\leqslant\int\varphi d\nu,\text{ }\forall\varphi\in
\mathcal{A}\text{.}%
\]

\end{definition}

Let $p\geqslant1,$ $\mu,$ $\nu\in P_{p}\left(  \mathbb{R}^{d}\right)  .$
Denote respectively by $P_{\mu\leqslant}$ and $P_{\leqslant\nu}$ the backward
and forward cone,%
\[
P_{\leqslant\nu}=\left\{  \eta\in P_{1}\left(  \mathbb{R}^{d}\right)
:\eta\leqslant_{cx}\nu\right\}  ,\text{ }P_{\mu\leqslant}=\left\{  \xi\in
P_{1}\left(  \mathbb{R}^{d}\right)  :\mu\leqslant_{cx}\xi\right\}  .
\]
To emphasize that the cones are built upon convex order, we can write
$P_{\leqslant\nu}^{cx},$ $P_{\mu\leqslant}^{cx}.$ But since we are dealing
exclusively with convex order, we have dropped those superscripts `$cx$' for
notational simplicity.

\begin{remark}
\label{rmk:cone_moment}With $\nu\in P_{p}\left(  \mathbb{R}^{d}\right)
,$\ the backward cone inherits $p$-th moment, $P_{\leqslant\nu}\subset
P_{p}\left(  \mathbb{R}^{d}\right)  .$ So $P_{\leqslant\nu}=P_{p}\left(
\mathbb{R}^{d}\right)  \cap P_{\leqslant\nu},$ hence%
\[
T_{c}\left(  \mu,P_{\leqslant\nu}\right)  =T_{c}\left(  \mu,P_{p}\left(
\mathbb{R}^{d}\right)  \cap P_{\leqslant\nu}\right)  .
\]
But with $\mu\in P_{p}\left(  \mathbb{R}^{d}\right)  ,$ we generally do not
have $P_{\mu\leqslant}\subset P_{p}\left(  \mathbb{R}^{d}\right)  .$ However,
we can show that the Wasserstein projection cost $T_{c}\left(  P_{\mu
\leqslant},\nu\right)  $ of any $\nu\in P_{p}\left(  \mathbb{R}^{d}\right)  $
onto the forward cone $P_{\mu\leqslant}$ is attainable by some probability of
$p$-th moment (see \cite[Theorem 5.3]{yh_yl_stochastic_order}), hence
\[
T_{c}\left(  P_{\mu\leqslant},\nu\right)  =T_{c}\left(  P_{p}\left(
\mathbb{R}^{d}\right)  \cap P_{\mu\leqslant},\nu\right)  .
\]
Therefore with $\mu,$ $\nu\in P_{p}\left(  \mathbb{R}^{d}\right)  $, the
previous definitions of $P_{\leqslant\nu},$ $P_{\mu\leqslant}$ are equivalent
to%
\[
P_{\leqslant\nu}=\left\{  \eta\in P_{p}\left(  \mathbb{R}^{d}\right)
:\eta\leqslant_{cx}\nu\right\}  ,\text{ }P_{\mu\leqslant}=\left\{  \xi\in
P_{p}\left(  \mathbb{R}^{d}\right)  :\mu\leqslant_{cx}\xi\right\}  .
\]

\end{remark}

\begin{definition}
\label{def:growth}Let $p\geqslant1.$ Unless stated otherwise, the cost
$c\left(  x,y\right)  $ is defined to have the form $c\left(  x,y\right)
=h\left(  x-y\right)  ,$ where $h:\mathbb{R}^{d}\mapsto\left[  0,\infty
\right)  $ is \textbf{strictly convex} such that$,$%
\begin{equation}
\inf_{\left\vert z\right\vert \geqslant a}\frac{h\left(  z\right)
}{\left\vert z\right\vert }>0\text{ and }\sup_{\left\vert z\right\vert
\geqslant a}\frac{h\left(  z\right)  }{\left\vert z\right\vert ^{p}}%
<\infty\text{ for some }a>0\text{.} \label{def:growth_eq1}%
\end{equation}

\end{definition}

\begin{definition}
For any $\varphi$, $\psi$, the $c$-transform and $\bar{c}$-transform are
respectively defined by%
\begin{equation}
Q_{c}\left(  \varphi\right)  \left(  x\right)  =\inf_{y\in\mathbb{R}^{d}%
}\left\{  \varphi\left(  y\right)  +c\left(  x,y\right)  \right\}  ,\text{
}Q_{\bar{c}}\left(  \psi\right)  \left(  y\right)  =\sup_{x\in\mathbb{R}^{d}%
}\left\{  \psi\left(  x\right)  -c\left(  x,y\right)  \right\}  .
\label{eq_QcQcbar}%
\end{equation}
A function as a result of $c$-transform (resp. $\bar{c}$-transform) is called
$c$-concave (resp. $\bar{c}$-concave).
\end{definition}

\section{Dual attainment and regularity}

For the moment, we put aside the requirement on $h$ of Definition
\ref{def:growth}. We ask: in order to apply the dualities \cite[Theorem 4.3,
Theorem 4.4]{yh_yl_stochastic_order} to a cost of the (not necessarily radial)
form $c\left(  x,y\right)  =h\left(  x-y\right)  ,$ what conditions are
required of $h$ $?$ Specifically we need to ensure that there exist $a>0$ and
$b\in\mathbb{R}$ such that
\[
F_{a}\left(  x,y\right)  \triangleq h\left(  x-y\right)  +a\left\vert
x\right\vert ^{p}-\frac{1}{a}\left\vert y\right\vert ^{p}\geqslant
b\,\text{for all }x,y\in\mathbb{R}^{d}.
\]
This condition looks a bit obscure. Previously, only radial costs are shown to
satisfy it \cite[Lemma 4.5]{yh_yl_stochastic_order}. The lemma below shows
that this obscure condition is in fact equivalent to a lower growth bound on
$h.$

\begin{lemma}
\label{lm:cost_equiv}Let $p\geqslant1,$ $h:\mathbb{R}^{d}\rightarrow
\mathbb{R}$ be convex. There exist $a>0$ and $b\in\mathbb{R}$ such that%
\begin{equation}
F_{a}\left(  x,y\right)  \geqslant b\,\text{for all }x,y\in\mathbb{R}^{d}
\label{lm:cost_equiv_eq1}%
\end{equation}
if and only if there exist $A>0$ and $B\in\mathbb{R}$ such that%
\begin{equation}
h\left(  z\right)  \geqslant A\left\vert z\right\vert ^{p}-B\,\text{for every
}z\in\mathbb{R}^{d}. \label{lm:cost_equiv_eq2}%
\end{equation}

\end{lemma}

\begin{proof}
\textbf{1}. For $p>1,$ suppose $\left(  \ref{lm:cost_equiv_eq1}\right)  $
holds. Let $z=x-y$. Then $y=x-z$, so we first compute%
\[
\inf_{x\in\mathbb{R}^{d}}\left(  a\left\vert x\right\vert ^{p}-\frac{1}%
{a}\left\vert x-z\right\vert ^{p}\right)  .
\]
If $a<1$, taking $z=0$ and $\left\vert x\right\vert \rightarrow\infty$ gives
$-\infty$. If $a=1$ and $z\neq0$, taking $x=-tz\left\vert z\right\vert ^{-1}$
also gives $-\infty$ as $t\rightarrow\infty$. Hence $a>1$ is necessary for
$\left(  \ref{lm:cost_equiv_eq1}\right)  $ to hold. For $a>1$, let%
\[
K_{p}\left(  a\right)  \triangleq\frac{a^{-1}}{\left(  1-a^{-2/\left(
p-1\right)  }\right)  ^{p-1}}.
\]
We claim that%
\begin{equation}
\inf_{x\in\mathbb{R}^{d}}\left(  a\left\vert x\right\vert ^{p}-\frac{1}%
{a}\left\vert x-z\right\vert ^{p}\right)  =-K_{p}\left(  a\right)  \left\vert
z\right\vert ^{p}. \label{lm:cost_equiv_eq3}%
\end{equation}
Indeed, writing $s=\left\vert x\right\vert $, $t=\left\vert z\right\vert $ and
using triangle inequality on $\left\vert x-z\right\vert $ we have%
\[
a\left\vert x\right\vert ^{p}-\frac{1}{a}\left\vert x-z\right\vert
^{p}\geqslant as^{p}-\frac{1}{a}\left(  s+t\right)  ^{p}.
\]
Equality in the triangle inequality is obtained by choosing $x$ opposite to
$z$. Therefore the minimization in $\left(  \ref{lm:cost_equiv_eq3}\right)  $
reduces to%
\[
\inf_{s\geqslant0}\left[  as^{p}-\frac{1}{a}\left(  s+t\right)  ^{p}\right]
.
\]
The minimizer is%
\[
s=\frac{\rho}{1-\rho}\,t,\text{ }\rho=a^{-2/\left(  p-1\right)  }.
\]
and%
\[
s+t=\frac{1}{1-\rho}t,\text{ }a\rho^{p-1}=\frac{\rho}{a}.
\]
So substitution gives $\left(  \ref{lm:cost_equiv_eq3}\right)  $.
Consequently,%
\[
\inf_{\substack{x,y\in\mathbb{R}^{d}\\x-y=z}}F_{a}\left(  x,y\right)
=h\left(  z\right)  -K_{p}\left(  a\right)  \left\vert z\right\vert ^{p}.
\]
Therefore $\left(  \ref{lm:cost_equiv_eq1}\right)  $ gives%
\[
h\left(  z\right)  \geqslant K_{p}\left(  a\right)  \left\vert z\right\vert
^{p}+b,
\]
which is $\left(  \ref{lm:cost_equiv_eq2}\right)  $. Conversely, suppose
$\left(  \ref{lm:cost_equiv_eq2}\right)  $ holds. Since%
\[
K_{p}\left(  a\right)  \rightarrow0\,\text{as }a\rightarrow\infty,
\]
choose $a>1$ such that $K_{p}\left(  a\right)  \leqslant A$. Then, using
$\left(  \ref{lm:cost_equiv_eq2}\right)  $ and $\left(
\ref{lm:cost_equiv_eq3}\right)  $,%
\[
F_{a}\left(  x,y\right)  \geqslant h\left(  x-y\right)  -K_{p}\left(
a\right)  \left\vert x-y\right\vert ^{p}\geqslant\left(  A-K_{p}\left(
a\right)  \right)  \left\vert x-y\right\vert ^{p}-B\geqslant-B.
\]
Thus we can take $b=-B$.

\textbf{2}. For $p=1,$ let $z=x-y$,%
\[
\inf_{y\in\mathbb{R}^{d}}\left(  a\left\vert z+y\right\vert -\frac{1}%
{a}\left\vert y\right\vert \right)
\]
is finite (for arbitrary $z$) only when $a\geqslant1$. In that case,
$\left\vert y\right\vert \leqslant\left\vert z+y\right\vert +\left\vert
z\right\vert $ gives%
\[
a\left\vert z+y\right\vert -\frac{1}{a}\left\vert y\right\vert \geqslant
\left(  a-\frac{1}{a}\right)  \left\vert z+y\right\vert -\frac{1}{a}\left\vert
z\right\vert \geqslant-\frac{1}{a}\left\vert z\right\vert .
\]
Equality is obtained at $y=-z$. Therefore%
\begin{equation}
\inf_{\substack{x,y\in\mathbb{R}^{d}\\x-y=z}}\left(  a\left\vert x\right\vert
-\frac{1}{a}\left\vert y\right\vert \right)  =-\frac{1}{a}\left\vert
z\right\vert . \label{lm:cost_equiv_eq4}%
\end{equation}
Hence if $\left(  \ref{lm:cost_equiv_eq1}\right)  $ holds, then%
\[
h\left(  z\right)  \geqslant\frac{1}{a}\left\vert z\right\vert +b.
\]
Conversely, if $h\left(  z\right)  \geqslant A\left\vert z\right\vert -B$,
choose%
\[
a\geqslant\max\left\{  1,A^{-1}\right\}  .
\]
Then $a^{-1}\leqslant A$, and $\left(  \ref{lm:cost_equiv_eq4}\right)  $ gives
$F_{a}\left(  x,y\right)  \geqslant-B$.
\end{proof}

\begin{remark}
Lemma \ref{lm:cost_equiv} remains valid for $0<p<1$, but we omit its proof
since its proof is similar. For convex $h$, $\left(  \ref{lm:cost_equiv_eq2}%
\right)  $ with $0<p\leqslant1$ is equivalent to the coercivity%
\[
h\left(  z\right)  \rightarrow\infty\text{ as }\left\vert z\right\vert
\rightarrow\infty.
\]
Thus, if $0<p\leqslant1,$ coercivity is sufficient and necessary for $\left(
\ref{lm:cost_equiv_eq1}\right)  $.
\end{remark}

\begin{theorem}
\label{thm:duality}Let $p\geqslant1$, $\mu,\nu\in P_{p}\left(  \mathbb{R}%
^{d}\right)  ,$ $c=h\left(  x-y\right)  $ satisfying Definition
\ref{def:growth}. Then the following hold.

$\left(  i\right)  $ The optimal backward projection measure exists. It is
unique if $\mu\in P_{p}^{ac}\left(  \mathbb{R}^{d}\right)  .$ Moreover the
backward duality holds,%
\[
T_{c}\left(  \mu,P_{\leqslant\nu}\right)  =D_{c}\left(  \mu,P_{\leqslant\nu
}\right)  \triangleq\sup_{\varphi\in\mathcal{A}\cap C_{b,p}}\left\{
\int_{\mathbb{R}^{d}}Q_{c}\left(  \varphi\right)  d\mu-\int_{\mathbb{R}^{d}%
}\varphi d\nu\right\}  .
\]
The function class over which the supremum is taken can be relaxed to
$\mathcal{A}\cap L^{1}\left(  d\nu\right)  .$

$\left(  ii\right)  $ The optimal forward projection measure exists. It is
unique if $\nu\in P_{p}^{ac}\left(  \mathbb{R}^{d}\right)  .$ Moreover, the
forward duality holds,%
\[
T_{c}\left(  P_{\mu\leqslant},\nu\right)  =D_{c}\left(  P_{\mu\leqslant}%
,\nu\right)  \triangleq\sup_{\psi\in\mathcal{A}\cap C_{b,p}}\left\{
\int_{\mathbb{R}^{d}}\psi d\mu-\int_{\mathbb{R}^{d}}Q_{\bar{c}}\left(
\psi\right)  d\nu\right\}  .
\]
The function class over which the supremum is taken can be restricted to%
\[
\left\{  \psi\in\mathcal{A}\cap C_{b,p}:Q_{\bar{c}}\left(  \psi\right)  \in
L^{1}\left(  d\nu\right)  \right\}  .
\]

\end{theorem}

\begin{proof}
First note that $\mathcal{A}$ and $\mathcal{A}\cap C_{b,p}$ define the same
convex order \cite[Lemma 5.2]{yh_yl_stochastic_order}.\ Under the growth
condition of Definition \ref{def:growth}$,$ $h$ satisfies $\left(
\ref{lm:cost_equiv_eq1}\right)  $ of Lemma \ref{lm:cost_equiv}. Thus
\cite[Theorem 4.3, Theorem 4.4]{yh_yl_stochastic_order} are applicable, the
optimal projections exist and dualities hold. The uniqueness under absolute
continuity follows from \cite[Theorem 4.11]{yh_yl_stochastic_order}. The
relaxation in backward duality follows from \cite[Remark 4.7]%
{yh_yl_stochastic_order}. As for the forward projection, since $\nu$ has the
$p$-th moment and $\psi\in C_{b,p},$ the term $\int\psi d\mu$\ is finite. The
convexity of $Q_{\bar{c}}\left(  \psi\right)  $ implies that $\int Q_{\bar{c}%
}\left(  \psi\right)  d\nu$ is lower bounded. Clearly those $\psi
\in\mathcal{A}\cap C_{b,p}$ for which $\int Q_{\bar{c}}\left(  \psi\right)
d\nu=\infty$ would not contribute to the supremum of $D_{c}\left(
P_{\mu\leqslant},\nu\right)  ,$ hence the supremum can equivalently be taken
over those $\psi\in\mathcal{A}\cap C_{b,p}$ for which $Q_{\bar{c}}\left(
\psi\right)  \in L^{1}\left(  d\nu\right)  .$
\end{proof}

The next theorem is a backward-forward swap property for the cost $c=h\left(
x-y\right)  $. In the case of $p$-cost $\left(  p\geqslant1\right)  ,$ it is
first proved in \cite[Corollary 4.4]{alfonsi2020sampling} using primal
formulation and then obtained by duality in \cite[Theorem 8.3]%
{yh_yl_stochastic_order}. Note that we do not need dual attainment at this stage.

\begin{theorem}
[\textbf{Backward-forward swap}]\label{thm:bf_swap}Let $p\geqslant1$ be an
integer, $\mu,$ $\nu\in P_{p}\left(  \mathbb{R}^{d}\right)  $. Then%
\[
T_{c}\left(  P_{\mu\leqslant},\nu\right)  =D_{c}\left(  P_{\mu\leqslant}%
,\nu\right)  =D_{c}\left(  \mu,P_{\leqslant\nu}\right)  =T_{c}\left(
\mu,P_{\leqslant\nu}\right)  .
\]
Moreover

$\left(  i\right)  $ if $\varphi\in\mathcal{A}\cap L^{1}\left(  d\nu\right)  $
is optimal for backward dual, then $Q_{c}\left(  \varphi\right)
\in\mathcal{A}\cap C_{b,p}$ is optimal for forward dual,

$\left(  ii\right)  $ if $\psi\in\mathcal{A}\cap C_{b,p}$ is optimal for
forward dual, then $Q_{\bar{c}}\left(  \psi\right)  \in\mathcal{A}\cap
L^{1}\left(  d\nu\right)  $ is optimal for backward dual.
\end{theorem}

\begin{proof}
\textbf{1}. Observe two simple facts. The first fact, if $\psi\in
\mathcal{A}\cap C_{b,p}$ and $Q_{\bar{c}}\left(  \psi\right)  \in L^{1}\left(
d\nu\right)  $, then clearly $Q_{\bar{c}}\left(  \psi\right)  \in
\mathcal{A}\cap L^{1}\left(  d\nu\right)  $ since $Q_{\bar{c}}\left(
\psi\right)  $ is convex. The second fact is sort of a converse to the first,
if $\varphi\in\mathcal{A}\cap L^{1}\left(  d\nu\right)  ,$ then $Q_{c}\left(
\varphi\right)  \in\mathcal{A}\cap C_{b,p}$ and $Q_{\bar{c}}\left(
Q_{c}\left(  \varphi\right)  \right)  \in L^{1}\left(  d\nu\right)  .$ To see
this recall (Definition \ref{def:the_class}) that any function in
$\mathcal{A}$ is finite-valued, hence $\varphi\left(  0\right)  $ is finite.
It follows that $Q_{c}\left(  \varphi\right)  $ is bounded from above by
$\varphi\left(  0\right)  +c\left(  x,0\right)  $ which grows no faster than
$\left\vert x\right\vert ^{p}$ (Definition \ref{def:growth}). Since
$Q_{c}\left(  \varphi\right)  $ is convex, it is thus a real-valued convex
function. Therefore $Q_{c}\left(  \varphi\right)  $ belongs to $\mathcal{A}%
\cap C_{b,p}.$ Finally, since $Q_{\bar{c}}\left(  Q_{c}\left(  \varphi\right)
\right)  $ is convex (hence supported by some linear function) and $Q_{\bar
{c}}\left(  Q_{c}\left(  \varphi\right)  \right)  \leqslant\varphi,$ hence
$Q_{\bar{c}}\left(  Q_{c}\left(  \varphi\right)  \right)  $ belongs to
$L^{1}\left(  d\nu\right)  .$

\textbf{2}. By Theorem \ref{thm:duality}, we have
\begin{equation}
D_{c}\left(  P_{\mu\leqslant},\nu\right)  =\sup_{\substack{\psi\in
\mathcal{A}\cap C_{b,p}\\Q_{\bar{c}}\left(  \psi\right)  \in L^{1}\left(
d\nu\right)  }}\left\{  \int_{\mathbb{R}^{d}}\psi d\mu-\int_{\mathbb{R}^{d}%
}Q_{\bar{c}}\left(  \psi\right)  d\nu\right\}  \label{thm:bf_swap_1}%
\end{equation}
and%
\begin{equation}
D_{c}\left(  \mu,P_{\leqslant\nu}\right)  =\sup_{\varphi\in\mathcal{A}\cap
L^{1}\left(  d\nu\right)  }\left\{  \int_{\mathbb{R}^{d}}Q_{c}\left(
\varphi\right)  d\mu-\int_{\mathbb{R}^{d}}\varphi d\nu\right\}  .
\label{thm:bf_swap_2}%
\end{equation}
On the one hand, by the first fact,%
\begin{align}
D_{c}\left(  \mu,P_{\leqslant\nu}\right)   &  \geqslant\sup_{_{\substack{\psi
\in\mathcal{A}\cap C_{b,p}\\Q_{\bar{c}}\left(  \psi\right)  \in L^{1}\left(
d\nu\right)  }}}\left\{  \int_{\mathbb{R}^{d}}Q_{c}\left(  Q_{\bar{c}}\left(
\psi\right)  \right)  d\mu-\int_{\mathbb{R}^{d}}Q_{\bar{c}}\left(
\psi\right)  d\nu\right\} \label{thm:bf_swap_3}\\
&  \geqslant\sup_{_{\substack{\psi\in\mathcal{A}\cap C_{b,p}\\Q_{\bar{c}%
}\left(  \psi\right)  \in L^{1}\left(  d\nu\right)  }}}\left\{  \int%
_{\mathbb{R}^{d}}\psi d\mu-\int_{\mathbb{R}^{d}}Q_{\bar{c}}\left(
\psi\right)  d\nu\right\}  =D_{c}\left(  P_{\mu\leqslant},\nu\right)
.\nonumber
\end{align}
On the other hand, by using the second fact,%
\begin{align}
D_{c}\left(  P_{\mu\leqslant},\nu\right)   &  \geqslant\sup_{\varphi
\in\mathcal{A}\cap L^{1}\left(  d\nu\right)  }\left\{  \int_{\mathbb{R}^{d}%
}Q_{c}\left(  \varphi\right)  d\mu-\int_{\mathbb{R}^{d}}Q_{\bar{c}}\left(
Q_{c}\left(  \varphi\right)  \right)  d\nu\right\} \label{thm:bf_swap_4}\\
&  \geqslant\sup_{\varphi\in\mathcal{A}\cap L^{1}\left(  d\nu\right)
}\left\{  \int_{\mathbb{R}^{d}}Q_{c}\left(  \varphi\right)  d\mu
-\int_{\mathbb{R}^{d}}\varphi d\nu\right\}  =D_{c}\left(  \mu,P_{\leqslant\nu
}\right)  .\nonumber
\end{align}
Combining the above, we obtain%
\[
D_{c}\left(  P_{\mu\leqslant},\nu\right)  =D_{c}\left(  \mu,P_{\leqslant\nu
}\right)  .
\]
Therefore we have prove that the inequalities in $\left(  \ref{thm:bf_swap_3}%
\right)  $ and $\left(  \ref{thm:bf_swap_4}\right)  $ are all equalities.

\textbf{3}. If $\varphi\in\mathcal{A}\cap L^{1}\left(  d\nu\right)  $ is
optimal for $D_{c}\left(  \mu,P_{\leqslant\nu}\right)  $, then $\left(
\ref{thm:bf_swap_4}\right)  $ reads as%
\[
D_{c}\left(  P_{\mu\leqslant},\nu\right)  \geqslant\int_{\mathbb{R}^{d}}%
Q_{c}\left(  \varphi\right)  d\mu-\int_{\mathbb{R}^{d}}Q_{\bar{c}}\left(
Q_{c}\left(  \varphi\right)  \right)  d\nu\geqslant\int_{\mathbb{R}^{d}}%
Q_{c}\left(  \varphi\right)  d\mu-\int_{\mathbb{R}^{d}}\varphi d\nu
=D_{c}\left(  \mu,P_{\leqslant\nu}\right)  .
\]
The outermost terms are equal, hence $Q_{c}\left(  \varphi\right)  $ is
optimal for forward dual. If $\psi\in\mathcal{A}\cap C_{b,p}$ is optimal for
$D_{c}\left(  P_{\mu\leqslant},\nu\right)  $,
\[
D_{c}\left(  P_{\mu\leqslant},\nu\right)  =\int_{\mathbb{R}^{d}}\psi d\mu
-\int_{\mathbb{R}^{d}}Q_{\bar{c}}\left(  \psi\right)  d\nu,
\]
then since the dual value is finite and $\psi$ is bounded by some linear
function of $\left\vert x\right\vert ^{p},$ we have $Q_{\bar{c}}\left(
\psi\right)  \in\mathcal{A}\cap L^{1}\left(  d\nu\right)  .$ Thus $\left(
\ref{thm:bf_swap_3}\right)  $ reads as%
\[
D_{c}\left(  \mu,P_{\leqslant\nu}\right)  \geqslant\int_{\mathbb{R}^{d}}%
Q_{c}\left(  Q_{\bar{c}}\left(  \psi\right)  \right)  d\mu-\int_{\mathbb{R}%
^{d}}Q_{\bar{c}}\left(  \psi\right)  d\nu\geqslant\int_{\mathbb{R}^{d}}\psi
d\mu-\int_{\mathbb{R}^{d}}Q_{\bar{c}}\left(  \psi\right)  d\nu=D_{c}\left(
P_{\mu\leqslant},\nu\right)  .
\]
This concludes that $Q_{\bar{c}}\left(  \psi\right)  $ is an optimal for
$D_{c}\left(  \mu,P_{\leqslant\nu}\right)  $.
\end{proof}

The following theorem shows that both the backward and forward dual problem
are attainable. A crucial step in the proof is a new strategy that allows us
to obtain local compactness of maximizing sequences.

\begin{theorem}
[\textbf{Dual attainment}]\label{thm:dual_att}Let $p\geqslant1$, $\mu,$
$\nu\in P_{p}\left(  \mathbb{R}^{d}\right)  $. Then

$\left(  i\right)  $ The backward dual $D_{c}\left(  \mu,P_{\leqslant\nu
}\right)  $ is attained in the class $\mathcal{A}\cap L^{1}\left(
d\nu\right)  $.

$\left(  ii\right)  $ The forward dual $D_{c}\left(  P_{\mu\leqslant}%
,\nu\right)  $ is attained in the class $\mathcal{A}\cap C_{b,p}$.

$\left(  iii\right)  $ Let $m_{\nu}=\int yd\nu.$ If $\varphi$ is an optimizer
for $D_{c}\left(  \mu,P_{\leqslant\nu}\right)  $ with $\varphi\left(  m_{\nu
}\right)  =0,$ then an optimizer $\psi$ for $D_{c}\left(  P_{\mu\leqslant}%
,\nu\right)  $ may be chosen so that $\psi=Q_{c}\left(  \varphi\right)  $ and
possesses the pointwise bound%
\[
-\kappa\left\vert x\right\vert -\alpha\leqslant\psi\left(  x\right)  \leqslant
h\left(  x-m_{\nu}\right)  ,\text{ }\forall x,
\]
where $\kappa,$ $\alpha>0$ are constants depending only on $m_{\nu},$ the cost
$c$ and the bounds of the dual.
\end{theorem}

\begin{proof}
\textbf{1}. We start by proving backward attainment. First observe that the
original form of the backward dual (Theorem \ref{thm:duality}) may be written
as%
\begin{equation}
D_{c}\left(  \mu,P_{\leqslant\nu}\right)  =\sup_{\substack{\varphi
\in\mathcal{A}\cap C_{b,p}\\\varphi=Q_{\bar{c}}\left(  Q_{c}\left(
\varphi\right)  \right)  ,\varphi\left(  m_{\nu}\right)  =0}}\left\{
\int_{\mathbb{R}^{d}}Q_{c}\left(  \varphi\right)  d\mu-\int_{\mathbb{R}^{d}%
}\varphi d\nu\right\}  . \label{thm:duality_bw}%
\end{equation}
To see this, it sufficies to note that for any $\varphi\in\mathcal{A}\cap
C_{b,p}$, $Q_{\bar{c}}\left(  Q_{c}\left(  \varphi\right)  \right)  $ is
convex and $Q_{\bar{c}}\left(  Q_{c}\left(  \varphi\right)  \right)
\leqslant\varphi,$ hence $Q_{\bar{c}}\left(  Q_{c}\left(  \varphi\right)
\right)  $ is admissible to the class $\mathcal{A}\cap C_{b,p}$ and thus can
be used as a candidate for supremum in the original form,%
\[
D_{c}\left(  \mu,P_{\leqslant\nu}\right)  \geqslant\sup_{\varphi\in
\mathcal{A}\cap C_{b,p}}\left\{  \int_{\mathbb{R}^{d}}Q_{c}\left(  Q_{\bar{c}%
}\left(  Q_{c}\left(  \varphi\right)  \right)  \right)  d\mu-\int%
_{\mathbb{R}^{d}}Q_{\bar{c}}\left(  Q_{c}\left(  \varphi\right)  \right)
d\nu\right\}
\]
But $Q_{c}\left(  Q_{\bar{c}}\left(  Q_{c}\left(  \varphi\right)  \right)
\right)  =Q_{c}\left(  \varphi\right)  $ and $Q_{\bar{c}}\left(  Q_{c}\left(
\varphi\right)  \right)  \leqslant\varphi$ for any $\varphi$. So the RHS is no
less than the original $D_{c}\left(  \mu,P_{\leqslant\nu}\right)  .\ $Also
noting that shifting $\varphi$ does not change the dual value. Hence $\left(
\ref{thm:duality_bw}\right)  $ holds. Let $\varphi_{n}\in\mathcal{A}\cap
C_{b,p}$ be a maximizing sequence such that $\varphi_{n}=Q_{\bar{c}}\left(
Q_{c}\left(  \varphi_{n}\right)  \right)  ,$ $\varphi_{n}\left(  m_{\nu
}\right)  =0.$ Then%
\begin{equation}
Q_{c}\left(  \varphi_{n}\right)  \left(  x\right)  =\inf_{y\in\mathbb{R}^{d}%
}\left\{  \varphi_{n}\left(  y\right)  +c\left(  x,y\right)  \right\}
\leqslant h\left(  x-m_{\nu}\right)  ,\text{ }\forall x.
\label{thm:dual_att_inq1}%
\end{equation}
Moreover by Jensen inequality%
\begin{equation}
\int_{\mathbb{R}^{d}}\varphi_{n}d\nu\geqslant\varphi_{n}\left(  \int%
_{\mathbb{R}^{d}}yd\nu\right)  =\varphi_{n}\left(  m_{\nu}\right)  =0.
\label{thm:dual_att_inq1.1}%
\end{equation}
But the sequence of total dual energies
\begin{equation}
\left\{  \int_{\mathbb{R}^{d}}Q_{c}\left(  \varphi_{n}\right)  d\mu
-\int_{\mathbb{R}^{d}}\varphi_{n}d\nu:n\geqslant1\right\}  ,
\label{thm:dual_att_inq2}%
\end{equation}
is bounded, hence%
\begin{equation}
\inf_{n}\int_{\mathbb{R}^{d}}Q_{c}\left(  \varphi_{n}\right)  d\mu
>-\infty\text{.} \label{thm:dual_att_inq3}%
\end{equation}
This indicates that there exists at least one point $x_{0}\in\mathbb{R}^{d}$
for which $Q_{c}\left(  \varphi_{n}\right)  \left(  x_{0}\right)  $ is bounded
from below along some subsequence. Indeed, suppose to the contrary that such a
point does not exist. Then consider all $2^{d}$ vertices of a closed
$d$-dimensional cube $K_{d}$ centering around the origin, we may therefore
find a subsequence $n_{l}$ $\left(  l\geqslant1\right)  $ such that%
\[
Q_{c}\left(  \varphi_{n_{l}}\right)  \left(  q\right)  \rightarrow
-\infty\text{ uniformly for all vertices }q\text{ of }K_{d}\text{ as
}l\rightarrow\infty.
\]
Precisely, for any $a<0,$ there is $l_{0}\geqslant1$ so that%
\begin{equation}
Q_{c}\left(  \varphi_{n_{l}}\right)  \left(  q\right)  \leqslant a\text{ for
all vertices }v\text{ of }K_{d}\text{ and }l\geqslant l_{0}.
\label{thm:dual_att_inq3.0}%
\end{equation}
Without loss of generality we assume that $\mu\left(  K_{d}\right)  >0$
(otherwise replace $K_{d}$ with a larger one)$.$ Note that any $x\in K_{d}$ is
a convex combination of the vertices of $K_{d}$. Since $Q_{c}\left(
\varphi_{n_{l}}\right)  $ is convex, any $Q_{c}\left(  \varphi_{n_{l}}\right)
\left(  x\right)  $ with $x\in K_{d}$ must be less than a convex combination
of the values on the vertices of $K_{d},$
\[
\left\{  Q_{c}\left(  \varphi_{n_{l}}\right)  \left(  q\right)  :q\text{ is a
vertex of }K_{d}\right\}  .
\]
The combination coefficients may possibly depend on $l,$ but they are not
important, the point is the values inside are controlled by those on the
vertices. Therefore using $\left(  \ref{thm:dual_att_inq1}\right)  ,$ $\left(
\ref{thm:dual_att_inq3.0}\right)  $ we have for $l\geqslant l_{0},$%
\begin{align*}
\int_{\mathbb{R}^{d}}Q_{c}\left(  \varphi_{n_{l}}\right)  d\mu &  =\int%
_{K_{d}}Q_{c}\left(  \varphi_{n_{l}}\right)  d\mu+\int_{\mathbb{R}%
^{d}\backslash K_{d}}Q_{c}\left(  \varphi_{n_{l}}\right)  d\mu\\
&  \leqslant\int_{K_{d}}Q_{c}\left(  \varphi_{n_{l}}\right)  d\mu
+\int_{\mathbb{R}^{d}\backslash K_{d}}h\left(  x-m_{\nu}\right)  d\mu\\
&  \leqslant a\mu\left(  K_{d}\right)  +\int_{\mathbb{R}^{d}\backslash K_{d}%
}h\left(  x-m_{\nu}\right)  d\mu\leqslant a+\int_{\mathbb{R}^{d}}h\left(
x-m_{\nu}\right)  d\mu.
\end{align*}
Since $a<0$ is arbitrary, this contradicts $\left(  \ref{thm:dual_att_inq3}%
\right)  .$ Thus we have proved that there exists $x_{0}\in\mathbb{R}^{d}$
such that along a subsequence (still denoted by $n_{l}$), the sequence
$\left\{  Q_{c}\left(  \varphi_{n_{l}}\right)  \left(  x_{0}\right)  \right\}
_{l}$ is bounded from below. This together with $\left(
\ref{thm:dual_att_inq1}\right)  $ shows that the convex function $Q_{c}\left(
\varphi_{n_{l}}\right)  $ is bounded from below by some cone, i.e., there are
$\kappa,$ $\alpha>0,$%
\begin{equation}
Q_{c}\left(  \varphi_{n_{l}}\right)  \left(  x\right)  \geqslant
-\kappa\left\vert x\right\vert -\alpha,\text{ }\forall x\text{, }l.
\label{thm:dual_att_inq3.1}%
\end{equation}
Note the derivation of the lower boundedness of $\left\{  Q_{c}\left(
\varphi_{n_{l}}\right)  \left(  x_{0}\right)  \right\}  _{l}$ only uses
$\left(  \ref{thm:dual_att_inq3}\right)  $ which depends only the total dual
energy bounds, and the upper bound $\left(  \ref{thm:dual_att_inq1}\right)  $
of $Q_{c}\left(  \varphi_{n_{l}}\right)  \left(  x\right)  $ depends only on
$m_{\nu}$ and the cost, therefore the constants $\kappa,$ $\alpha$ are
determined by on the measure $\nu,$ the cost and the bounds of the dual.
Combining $\left(  \ref{thm:dual_att_inq3.1}\right)  $ with $\left(
\ref{thm:dual_att_inq1}\right)  ,$%
\begin{equation}
-\kappa\left\vert x\right\vert -\alpha\leqslant Q_{c}\left(  \varphi_{n_{l}%
}\right)  \left(  x\right)  \leqslant h\left(  x-m_{\nu}\right)  ,\text{
}\forall x\text{, }l, \label{thm:dual_att_inq3.1.1}%
\end{equation}
i.e., $Q_{c}\left(  \varphi_{n_{l}}\right)  \left(  x\right)  $ is a bounded
sequence for each fixed $x\in\mathbb{R}^{d}.$ Thus there is a convex function
$\varrho$ such that, up to a subsequence,%
\[
Q_{c}\left(  \varphi_{n_{l}}\right)  \left(  x\right)  \rightarrow
\varrho\left(  x\right)  ,\text{ }\forall x.
\]
Since $\varphi_{n_{l}}=Q_{\bar{c}}\left(  Q_{c}\left(  \varphi_{n_{l}}\right)
\right)  ,$%
\[
\varphi_{n_{l}}\left(  y\right)  \geqslant Q_{c}\left(  \varphi_{n_{l}%
}\right)  \left(  x\right)  -h\left(  x-y\right)  ,\text{ }\forall x,y.
\]
We have%
\[
\liminf_{l\rightarrow\infty}\varphi_{n_{l}}\left(  y\right)  \geqslant
\varrho\left(  x\right)  -h\left(  x-y\right)  ,\text{ }\forall x,y.
\]
Hence%
\begin{equation}
\liminf_{l\rightarrow\infty}\varphi_{n_{l}}\left(  y\right)  \geqslant
Q_{\bar{c}}\left(  \varrho\right)  \left(  y\right)  ,\text{ }\forall y.
\label{thm:dual_att_inq3.2}%
\end{equation}
Using Fatou's lemma%
\[
\liminf_{l\rightarrow\infty}\int_{\mathbb{R}^{d}}\varphi_{n_{l}}d\nu
\geqslant\int_{\mathbb{R}^{d}}\liminf_{l\rightarrow\infty}\varphi_{n_{l}%
}\left(  y\right)  d\nu\geqslant\int_{\mathbb{R}^{d}}Q_{\bar{c}}\left(
\varrho\right)  \left(  y\right)  d\nu.
\]
But from $\left(  \ref{thm:dual_att_inq2}\right)  \left(
\ref{thm:dual_att_inq3.1.1}\right)  ,$ we see that $\int\varphi_{n_{l}}d\nu$
is bounded. Thus $\int Q_{\bar{c}}\left(  \varrho\right)  \left(  y\right)
d\nu<\infty$, hence $Q_{\bar{c}}\left(  \varrho\right)  \in\mathcal{A}\cap
L^{1}\left(  d\nu\right)  .$ It follows that%
\begin{align*}
D_{c}\left(  \mu,P_{\leqslant\nu}\right)   &  =\lim_{l\rightarrow\infty
}\left\{  \int_{\mathbb{R}^{d}}Q_{c}\left(  \varphi_{n_{l}}\right)  d\mu
-\int_{\mathbb{R}^{d}}\varphi_{n_{l}}d\nu\right\} \\
&  \leqslant\limsup_{l\rightarrow\infty}\int_{\mathbb{R}^{d}}Q_{c}\left(
\varphi_{n_{l}}\right)  d\mu-\liminf_{l\rightarrow\infty}\int_{\mathbb{R}^{d}%
}\varphi_{n_{l}}d\nu\\
&  \leqslant\int_{\mathbb{R}^{d}}\varrho d\mu-\int_{\mathbb{R}^{d}}Q_{\bar{c}%
}\left(  \varrho\right)  d\nu\\
&  \leqslant\int_{\mathbb{R}^{d}}Q_{c}\left(  Q_{\bar{c}}\left(
\varrho\right)  \right)  d\mu-\int_{\mathbb{R}^{d}}Q_{\bar{c}}\left(
\varrho\right)  d\nu\leqslant D_{c}\left(  \mu,P_{\leqslant\nu}\right)
\end{align*}
The last inequality follows from Theorem \ref{thm:duality} $\left(  i\right)
$, since $Q_{\bar{c}}\left(  \varrho\right)  \in\mathcal{A}\cap L^{1}\left(
d\nu\right)  $. This shows that $Q_{\bar{c}}\left(  \varrho\right)  $ is
optimal for the backward dual.

\textbf{2}. The optimizer of the forward dual, now attainable in light of
Theorem \ref{thm:bf_swap}, is given by%
\[
\psi=Q_{c}\left(  Q_{\bar{c}}\left(  \varrho\right)  \right)  .
\]

\textbf{3}. It is worth noting that in Step 1, the constants $\kappa,$
$\alpha$ in $\left(  \ref{thm:dual_att_inq3.1}\right)  $ depend only on the
lower bound $\inf_{n}\int Q_{c}\left(  \varphi_{n}\right)  d\mu,$ which in
turn, upon having $\int\varphi_{n}d\nu\geqslant0,$ depends only on the bound
of $\left(  \ref{thm:dual_att_inq2}\right)  $. The result is that $\kappa$
depends only on the bound of $\left(  \ref{thm:dual_att_inq2}\right)  $.
Another thing is that we have from $\left(  \ref{thm:dual_att_inq3.2}\right)
$ that $Q_{\bar{c}}\left(  \varrho\right)  \left(  m_{\nu}\right)
\leqslant\varphi_{n}\left(  m_{\nu}\right)  =0.$ Combing it with $\left(
\ref{thm:dual_att_inq1}\right)  ,$ we get%
\[
Q_{c}\left(  Q_{\bar{c}}\left(  \varrho\right)  \right)  \left(  x\right)
\leqslant h\left(  x-m_{\nu}\right)  ,\text{ }\forall x.
\]
But from $\left(  \ref{thm:dual_att_inq3.1}\right)  ,$ $\varrho\geqslant
-\kappa\left\vert x\right\vert -\alpha,$ $\forall x.$ Hence, noting
$\varrho\leqslant Q_{c}\left(  Q_{\bar{c}}\left(  \varrho\right)  \right)  ,$
we get the bound for $\psi,$%
\begin{equation}
-\kappa\left\vert x\right\vert -\alpha\leqslant\psi\leqslant h\left(
x-m_{\nu}\right)  ,\text{ }\forall x. \label{rmk:thm_dual_att_inq1}%
\end{equation}

\end{proof}

Theorem \ref{thm:dual_att} indicates that\ the optimal dual potential of
backward projection generally has less regularity than forward projection. The
following theorem shows that in the case of convex Lipschitz cost, both
backward and forward optimal dual potentials have the same regularity. In
fact, we show that they are Lipschitz.

\begin{theorem}
[\textbf{Lipschitz dual attainment}]\label{thm:dual_att_lip_cost}Let $\mu
,\nu\in P_{1}\left(  \mathbb{R}^{d}\right)  $. Assume that $h$ is convex,
$L$-Lipschitz with constant $L>0$ and the cost $c$ is given by $h\left(
x-y\right)  $. Then $D_{c}\left(  \mu,P_{\leqslant\nu}\right)  $ and
$D_{c}\left(  P_{\mu\leqslant},\nu\right)  $ are attained in the class
$\mathcal{A}\cap Lip$. Moreover both optimal potentials are $L$-Lipschitz.
\end{theorem}

\begin{proof}
\textbf{1}. First consider backward attainment. As in Step 1 of Theorem
\ref{thm:dual_att}, we write $m_{\nu}=\int yd\nu$. Let $\varphi_{n}%
\in\mathcal{A}\cap C_{b,p}$ be a maximizing sequence of $D_{c}\left(
\mu,P_{\leqslant\nu}\right)  $ such that $\varphi_{n}=Q_{\bar{c}}\left(
Q_{c}\left(  \varphi_{n}\right)  \right)  ,$ $\varphi_{n}\left(  m_{\nu
}\right)  =0.$ Then $Q_{c}\left(  \varphi_{n}\right)  \left(  x\right)
\leqslant h\left(  x-m_{\nu}\right)  ,$ $\forall x.$ This shows that
$Q_{c}\left(  \varphi_{n}\right)  $ is pointwise bounded from above, which
readily produces a pointwise upper bound for $\varphi_{n}$. Indeed,%
\begin{align*}
\varphi_{n}\left(  y\right)   &  =Q_{\bar{c}}\left(  Q_{c}\left(  \varphi
_{n}\right)  \right)  \left(  y\right)  =\sup_{x\in\mathbb{R}^{d}}\left\{
Q_{c}\left(  \varphi_{n}\right)  \left(  x\right)  -h\left(  \left\vert
x-y\right\vert \right)  \right\} \\
&  \leqslant\sup_{x\in\mathbb{R}^{d}}\left\{  h\left(  x-m_{\nu}\right)
-h\left(  x-y\right)  \right\}  \leqslant L\left\vert y-m_{\nu}\right\vert
,\text{ }\forall y.
\end{align*}
We may now proceed as in Theorem \ref{thm:dual_att} and obtain at some point
that, up to a subsequence,%
\[
Q_{c}\left(  \varphi_{n}\right)  \left(  x\right)  \rightarrow\varrho\left(
x\right)  ,\text{ }\forall x,\text{ for some convex }\varrho,
\]
and%
\[
Q_{\bar{c}}\left(  \varrho\right)  \left(  y\right)  \leqslant\liminf
_{n\rightarrow\infty}\varphi_{n}\left(  y\right)  ,\text{ }\forall y.
\]
Hence the optimal potential $Q_{\bar{c}}\left(  \varrho\right)  $ for backward
dual is bounded from above by a $L$-Lipschitz function,%
\[
Q_{\bar{c}}\left(  \varrho\right)  \left(  y\right)  \leqslant L\left\vert
y-m_{\nu}\right\vert ,\text{ }\forall y.
\]
Since $Q_{\bar{c}}\left(  \varrho\right)  $ is convex , it must be $L$-Lipschitz.

\textbf{2}. Using Theorem \ref{thm:dual_att} $\left(  ii\right)  ,$ the
optimal potential $\psi\left(  x\right)  $ satisfies
\[
\psi\left(  x\right)  \leqslant h\left(  x-m_{\nu}\right)  ,\text{ }\forall
x,
\]
Once again since $\psi$ is convex, it must be $L$-Lipschitz. The conclusion follows.
\end{proof}

If $\left\Vert \cdot\right\Vert $ is a norm on $\mathbb{R}^{d},$ $c\left(
x,y\right)  =\left\Vert x-y\right\Vert ,$ then for any $\phi$ that is
$1$-Lipschitz w.r.t. $\left\Vert \cdot\right\Vert $, both $Q_{c}\left(
\phi\right)  $ and $Q_{\bar{c}}\left(  \phi\right)  $ are $1$-Lipschitz w.r.t.
$\left\Vert \cdot\right\Vert $. Moreover%
\[
\phi=Q_{c}\left(  \phi\right)  =Q_{\bar{c}}\left(  \phi\right)  .
\]
Also note that if a convex function is upper bounded by $\left\Vert
\cdot\right\Vert ,$ then it must be $1$-Lipschitz w.r.t. $\left\Vert
\cdot\right\Vert .$ Therefore following the lines of Theorem
\ref{thm:dual_att_lip_cost}, we get the Rubinstein duality for Wasserstein
projections and its attainment.

\begin{theorem}
[\textbf{Rubinstein duality}]Let $\mu,$ $\nu\in P_{1}\left(  \mathbb{R}%
^{d}\right)  $. Assume that $\left\Vert \cdot\right\Vert $ is a norm on
$\mathbb{R}^{d},$ $c\left(  x,y\right)  =\left\Vert x-y\right\Vert $. Then%
\[
D_{c}\left(  \mu,P_{\leqslant\nu}\right)  =D_{c}\left(  P_{\mu\leqslant}%
,\nu\right)  =\sup_{\varphi\in\mathcal{A},\text{ }1\text{-}\left\Vert
\cdot\right\Vert \text{-Lip}}\int_{\mathbb{R}^{d}}\varphi d\left(  \mu
-\nu\right)  .
\]
The supremum is attained by some convex $1$-Lipschitz function w.r.t.
$\left\Vert \cdot\right\Vert $.
\end{theorem}

One usefulness of Theorem \ref{thm:dual_att_lip_cost} is for the compactly
supported case. The following gives another proof of \cite[P455, line
2]{gozlan2020mixture} with slight improvement of the constraint set of the
infimum in $Q_{c}$ from $\left\{  y:\left\vert y-x\right\vert \leqslant
4R\right\}  .$

\begin{corollary}
\label{cor:dual_att_bd_spp}If $\mu,$ $\nu\in P\left(  \mathbb{R}^{d}\right)  $
have compact supports contained in $B_{R}=\left\{  \left\vert x\right\vert
\leqslant R\right\}  ,$ $c\left(  x,y\right)  =h\left(  x-y\right)  $ with $h$
satisfying Definition \ref{def:growth}. Then the backward projection has a
$L$-Lipschitz optimal dual potential where%
\[
L=\sup\left\{  \left\vert g\right\vert :g\in\partial h\left(  z\right)
,\,\left\vert z\right\vert \leqslant2R\right\}  .
\]

\end{corollary}

\begin{proof}
Let $\bar{\mu}$ be optimal projection of $\mu$, $\pi$ the optimal coupling
between them and $\varphi$ the optimal dual potential obtained in Theorem
\ref{thm:dual_att}. At optimality,%
\[
Q_{c}\left(  \varphi\right)  \left(  x\right)  -\varphi\left(  y\right)
=h\left(  x-y\right)  ,\text{ }\pi\text{-a.e. }x,\text{ }y.
\]
Since both marginals of $\pi$ are contained $B_{R},$ any transport between
$\mu$ and its projection travels a distance no more than $2R.$ Consequently we
can modify the cost in a way that it becomes global Lipschitz while covering
the active transport region so that the optimal projection is not affected:
let%
\[
\bar{h}\left(  z\right)  =\inf_{y\in\mathbb{R}^{d}}\left[  h\left(  y\right)
+L\left\vert z-y\right\vert \right]  ,
\]
Then $\bar{h}$ is global Lipschitz with constant $L$ and equals $h$ whenever
$\left\vert z\right\vert \leqslant2R$. Now instead of%
\[
Q_{c}\left(  \varphi\right)  \left(  x\right)  =\inf_{y\in\mathbb{R}^{d}%
}\left\{  \varphi\left(  y\right)  +h\left(  x-y\right)  \right\}  ,\text{
}\forall x\in B_{R},
\]
we can write%
\[
Q_{c}\left(  \varphi\right)  \left(  x\right)  =\inf_{y:\left\vert
y-x\right\vert \leqslant2R}\left\{  \varphi\left(  y\right)  +\bar{h}\left(
x-y\right)  \right\}  ,\text{ }\forall x\in B_{R}.
\]
Moreover the original projection problem is equivalent to the one that has
$\bar{h}\left(  x-y\right)  $ as its cost function. Therefore Theorem
\ref{thm:dual_att_lip_cost} ensures the existence of a $L$-Lipschitz optimal
dual potential$.$
\end{proof}

We will hereafter call an optimizer $\varphi$ of $D_{c}\left(  \mu
,P_{\leqslant\nu}\right)  $ an optimal dual potential, while calling
$Q_{c}\left(  \varphi\right)  $ its optimal \textbf{transport potential}.
Together we call $\left(  Q_{c}\left(  \varphi\right)  ,\varphi\right)  $ an
optimal \textbf{dual potential pair}. Similarly, in forward dual, an optimizer
$\psi$ of $D_{c}\left(  P_{\mu\leqslant},\nu\right)  $ is called an optimal
dual potential, $Q_{\bar{c}}\left(  \psi\right)  $ an optimal transport
potential and $\left(  \psi,Q_{\bar{c}}\left(  \psi\right)  \right)  $ its
optimal dual potential pair.

Combining the previous theorems gives the existence of conjugate optimal dual
potential pairs for both backward and forward dual. If $\bar{\mu}$ is an
optimal backward projection of $T_{c}\left(  \mu,P_{\leqslant\nu}\right)  ,$
$\bar{\nu}$ an optimal forward projection of $T_{c}\left(  P_{\mu\leqslant
},\nu\right)  ,$ then the conjugate relation shows that the transport between
$\mu,$ $\bar{\mu}$ and between $\bar{\nu}$, $\nu$ are suppled by the same
optimal dual potential pair $\left(  Q_{c}\left(  \varphi\right)
,\varphi\right)  =\left(  \psi,Q_{\bar{c}}\left(  \psi\right)  \right)  $.

\begin{theorem}
[\textbf{Backward-forward conjugate}]\label{thm:conj_opt_pair}Let
$p\geqslant1$, $\mu,$ $\nu\in P_{p}\left(  \mathbb{R}^{d}\right)  $. Write
$S=\left(  \mathcal{A}\cap C_{b,p}\right)  \times\left(  \mathcal{A}\cap
L^{1}\left(  d\nu\right)  \right)  .$

$\left(  i\right)  $ There exists an optimal dual potential pair $\left(
Q_{c}\left(  \varphi\right)  ,\varphi\right)  \in S$ for $D_{c}\left(
\mu,P_{\leqslant\nu}\right)  $ such that%
\[
\varphi\left(  y\right)  =Q_{\bar{c}}\left(  Q_{c}\left(  \varphi\right)
\right)  \left(  y\right)  ,\nu\text{-a.e. }y.
\]

$\left(  ii\right)  $ There exists an optimal dual potential pair $\left(
\psi,Q_{\bar{c}}\left(  \psi\right)  \right)  \in S$ for $D_{c}\left(
P_{\mu\leqslant},\nu\right)  $ such that%
\[
\psi\left(  x\right)  =Q_{c}\left(  Q_{\bar{c}}\left(  \psi\right)  \right)
\left(  x\right)  ,\text{ }\mu\text{-a.e. }x.
\]
In particular, we may always modify the dual potential pair $\left(
\psi,\varphi\right)  \in S$ in a way that it is optimal for both backward and
forward dual, and it holds everywhere that%
\[
\varphi=Q_{\bar{c}}\left(  \psi\right)  ,\text{ }\psi=Q_{c}\left(
\varphi\right)  .
\]

\end{theorem}

\begin{proof}
We only prove the backward case. Indeed, since $\varphi\in\mathcal{A}\cap
L^{1}\left(  d\nu\right)  $ is an optimizer of the backward dual, using
Theorem \ref{thm:bf_swap} twice, we see that $Q_{c}\left(  \varphi\right)  $
is optimal for the forward dual and thus $Q_{\bar{c}}\left(  Q_{c}\left(
\varphi\right)  \right)  $ is optimal for the backward dual. Applying
$Q_{c}\left(  \cdot\right)  $ on this new optimizer $Q_{\bar{c}}\left(
Q_{c}\left(  \varphi\right)  \right)  $ gives us another optimal potential
pair for the backward dual. But $Q_{c}\left(  Q_{\bar{c}}\left(  Q_{c}\left(
\varphi\right)  \right)  \right)  =Q_{c}\left(  \varphi\right)  ,$ hence the
new optimal potential pair becomes $\left(  Q_{c}\left(  \varphi\right)
,Q_{\bar{c}}\left(  Q_{c}\left(  \varphi\right)  \right)  \right)  $. Let
$\pi$ be an optimal coupling between $\mu$ and an optimal solution of
$T_{c}\left(  \mu,P_{\leqslant\nu}\right)  $ (i.e. an optimal projection onto
$P_{\leqslant\nu}$)$.$ Respectively using complementary slackness for the two
pairs,%
\[
Q_{c}\left(  \varphi\right)  \left(  x\right)  -\varphi\left(  y\right)
=c\left(  x,y\right)  ,\text{ }\pi\text{-a.e. }\left(  x,y\right)
\]
and%
\[
Q_{c}\left(  \varphi\right)  \left(  x\right)  -Q_{\bar{c}}\left(
Q_{c}\left(  \varphi\right)  \right)  \left(  y\right)  =c\left(  x,y\right)
,\text{ }\pi\text{-a.e. }\left(  x,y\right)
\]
Therefore we must have%
\[
\varphi\left(  y\right)  =Q_{\bar{c}}\left(  Q_{c}\left(  \varphi\right)
\right)  \left(  y\right)  ,\text{ }\nu\text{-a.e. }y.
\]
This does not directly say that $\varphi=Q_{\bar{c}}\left(  Q_{c}\left(
\varphi\right)  \right)  $ everywhere. But, since $\varphi$ appears in the
duality formula as an integrand w.r.t. $\nu,$ replacing $\varphi$ with
$Q_{\bar{c}}\left(  Q_{c}\left(  \varphi\right)  \right)  $ will not change
the dual value, hence we can always let $\varphi$ equal $Q_{\bar{c}}\left(
Q_{c}\left(  \varphi\right)  \right)  $ everywhere.
\end{proof}

\section{Uniqueness of the optimal dual transport
potential\label{sec:uniq_potential}}

When speaking of optimal dual potential of $D_{c}\left(  \mu,P_{\leqslant\nu
}\right)  $, we mean a function $\varphi\in\mathcal{A}\cap C_{b,p}$ that
maximizes the dual. This is natural in light of the construction of the
projection in the convex order, since we are selecting maximizers from the
class that enforces the convex order relationship. But the actual optimal
transport work from $\mu$ to its optimal projection, say $\bar{\mu},$ on
$P_{\leqslant\nu}$ is done by $Q_{c}\left(  \varphi\right)  .$ Writing
\[
\psi\left(  x\right)  =Q_{c}\left(  \varphi\right)  \left(  x\right)  ,
\]
we have%
\begin{equation}
\bar{\mu}=T_{\#}\mu,\text{ where }T\left(  x\right)  =x-\nabla h^{\ast}\left(
\nabla\psi\left(  x\right)  \right)  ,\text{ a.e. }x.
\label{eq:unique_proj_meas}%
\end{equation}
Similarly for forward dual.

To deal with the uniqueness of optimal dual transport potential, we first have
to consider the uniqueness of the projection measure. In general the backward
(resp. forward) projection measure is unique whenever $\mu$ (resp. $\nu$) is
absolutely continuous w.r.t. the Lebesgue measure (Theorem \ref{thm:duality}).
For $p$-cost $\left\vert x-y\right\vert ^{p}$ ($p\geqslant1$), the backward
projection enjoy better property that even if $\mu$ is not absolutely
continuous, the backward projection measure is still unique \cite[Theorem
2.1]{alfonsi2020sampling}. As shown in the quadratic cost case, this indeed is
an interaction between the convexity of the cost and the convexity along
generalized geodesic of the backward cone \cite[Proposition 1.1]%
{gozlan2020mixture}. Therefore this property continues to hold\ provided the
cost has enough convexity.

\begin{definition}
[$\omega$-\textbf{uniformly convex}]\label{def:omeg_uni_cvx}A differentiable
$g$ is said to be $\omega$-uniformly convex if there exists a modulus $\omega$
such that for any points $a,b\in\mathbb{R}^{d}$, $h$ satisfies the first-order
inequality:
\begin{equation}
g\left(  b\right)  \geqslant g\left(  a\right)  +\nabla g\left(  a\right)
\cdot\left(  b-a\right)  +\omega\left(  \left\vert b-a\right\vert \right)  .
\label{def:omeg_uni_cvx_eq1}%
\end{equation}

\end{definition}

A modulus $\omega:\left[  0,\infty\right)  \mapsto\left[  0,\infty\right)  $
is typically an increasing function with $\omega\left(  0\right)  =0$ and
$\omega\left(  s\right)  >0$ for $s>0$. For any $x,y\in\mathbb{R}^{d},$
$t\in\left[  0,1\right]  ,$ evaluate $\left(  \ref{def:omeg_uni_cvx_eq1}%
\right)  $ at $a=\left(  1-t\right)  x+ty$ and $b=x$ gives%
\[
g\left(  x\right)  \geqslant g\left(  \left(  1-t\right)  x+ty\right)
+t\nabla g\left(  \left(  1-t\right)  x+ty\right)  \cdot\left(  x-y\right)
+\omega\left(  t\left\vert x-y\right\vert \right)  ,
\]
while at $a=\left(  1-t\right)  x+ty$ and $b=y$ gives
\[
g\left(  y\right)  \geqslant g\left(  \left(  1-t\right)  x+ty\right)
+\left(  1-t\right)  \nabla g\left(  \left(  1-t\right)  x+ty\right)
\cdot\left(  y-x\right)  +\omega\left(  \left(  1-t\right)  \left\vert
x-y\right\vert \right)  .
\]
Multiplying the former inequality by $1-t$ and the latter by $t$ yields the
generalized uniform convexity,%
\[
g\left(  \left(  1-t\right)  x+ty\right)  \leqslant\left(  1-t\right)
g\left(  x\right)  +tg\left(  y\right)  -\left[  \left(  1-t\right)
\omega\left(  t\left\vert x-y\right\vert \right)  +t\omega\left(  \left(
1-t\right)  \left\vert x-y\right\vert \right)  \right]  .
\]
To recover the familiar uniform convexity inequality, set $g\left(  x\right)
=\left\vert x\right\vert ^{p}$ for $p\geqslant2$,
\[
\left\vert \left(  1-t\right)  y+tz-x\right\vert ^{p}\leqslant\left(
1-t\right)  \left\vert y-x\right\vert ^{p}+t\left\vert z-x\right\vert
^{p}-2^{2-p}t\left(  1-t\right)  \left\vert y-z\right\vert ^{p}.
\]

As a result of Jensen inequality, the backward cone $P_{\leqslant\nu}$ is
convex along generalized geodesic. Combining an argument analogous to the
quadratic case \cite[Lemma 9.2.1]{ambrosio2005gradient} shows that (if $c$ is
$\omega$-uniformly convex) $T_{c}\left(  \mu,\cdot\right)  $ is strictly
convex along generalized geodesic of the backward cone $P_{\leqslant\nu}$.
This readily leads to the uniqueness of backward projection measure for
$T_{c}\left(  \mu,P_{\leqslant\nu}\right)  $ without requiring $\mu$ to be
absolutely continuous. Thus, if $c$ is the $p$-cost $\left\vert x-y\right\vert
^{p},$ then the uniqueness of projection measure follows whenever
$p\geqslant2.$

Once the uniqueness of projection measure is confirmed, we can proceed to the
uniqueness of optimal dual potential. For optimal transport, there are already
effort into this, e.g. \cite{staudt2025uniqueness}\cite{ford2026quantitative}.
But since it is not possible to make assumptions on the projected measures, we
cannot directly apply the existing results of optimal transport. We give two
sufficient conditions that will be useful in our context. The first is a
natural condition from the classical setup \cite{del2024central}. The second
needs stronger connectedness but has greater potential to generalize. First we
state the assumptions we need. Let $\xi\in P\left(  \mathbb{R}^{d}\right)  $,
the property (\textbf{Spt)} is satisfied if either one of the following holds.

\begin{itemize}
\item[(\textbf{Spt-a)}] The interior $\operatorname*{int}\left(
\operatorname*{supp}\left(  \xi\right)  \right)  $ of the support is
connected, $\xi$ does not charge the boundary $\partial\left(
\operatorname*{supp}\left(  \xi\right)  \right)  $, and $\xi$ has a density
which is $\mathcal{L}^{d}$-a.e. \textit{strictly positive} on
$\operatorname*{int}\left(  \operatorname*{supp}\left(  \xi\right)  \right)
$\textit{.}

\item[(\textbf{Spt-b)}] The support $\operatorname*{supp}\left(  \xi\right)  $
is rectifiably-connected in the sense that any two points in
$\operatorname*{supp}\left(  \xi\right)  $ can be connected via a rectifiable
curve in $\operatorname*{supp}\left(  \xi\right)  .$
\end{itemize}

\begin{lemma}
\label{lm:uniq1}Let $p\geqslant1$, $\mu,\nu\in P_{p}\left(  \mathbb{R}%
^{d}\right)  .$ Suppose that $\mu$ satisfies (\textbf{Spt-a)} or $\nu$
satisfies (\textbf{Spt-a)}. Then $D_{c}\left(  \mu,P_{\leqslant\nu}\right)  $
has a unique dual potential pair up to a translation constant. The same
conclusion holds for forward dual $D_{c}\left(  P_{\mu\leqslant},\nu\right)  $.
\end{lemma}

\begin{proof}
\textbf{1}. Assume that $\mu$ satisfies (\textbf{Spt-a)}. The assumption tells
us $\mu=\rho dx$ with $\rho\left(  x\right)  >0$ for $\mathcal{L}^{d}$-a.e.
$x\in\operatorname*{int}\left(  \operatorname*{supp}\left(  \xi\right)
\right)  .$ Let $\pi$ be an optimal coupling between $\mu$ and its optimal
projection for $T_{c}\left(  \mu,P_{\leqslant\nu}\right)  .$ Suppose that both
$\varphi_{i}\in\mathcal{A}\cap L^{1}\left(  d\nu\right)  $ are optimal dual
potentials , $i=1,2.$ The optimal dual transport potentials are given by
$\psi_{i}=Q_{c}\left(  \varphi_{i}\right)  $. By complementary slackness
\begin{equation}
\psi_{i}\left(  x\right)  -\varphi_{i}\left(  y\right)  =c\left(  x,y\right)
,\text{ }\pi\text{-a.e. }\left(  x,y\right)  . \label{lm:uniq1_eq1}%
\end{equation}
Owing to Theorem \ref{thm:bf_swap}, $\psi_{i}\in\mathcal{A}\cap C_{b,p}.$
Hence $\psi_{i}\in L^{1}\left(  d\mu\right)  $. Since $\psi_{i}$ is convex,
$\partial\left(  \operatorname*{dom}\left(  \psi_{i}\right)  \right)  $ is
$\mathcal{L}^{d}$-negligible. These combined imply that, $\mu\left(
\operatorname*{int}\left(  \operatorname*{dom}\left(  \psi_{i}\right)
\right)  \right)  =1.$ Since $\psi_{i}$ is $\mathcal{L}^{d}$-a.e.
differentiable in $\operatorname*{int}\left(  \operatorname*{dom}\left(
\psi_{i}\right)  \right)  ,$ hence by the absolute continuity, $\nabla\psi
_{i}\left(  x\right)  $ exists for $\mu$-a.e. $x.$ From this we can conclude
that%
\[
\nabla\psi_{i}\left(  x\right)  =\nabla c\left(  x,y\right)  ,\text{ }%
\pi\text{-a.e. }\left(  x,y\right)  .
\]
Thus%
\[
\nabla\psi_{1}\left(  x\right)  =\nabla\psi_{2}\left(  x\right)  ,\text{ }%
\pi\text{-a.e. }\left(  x,y\right)  .
\]
Since the $x$-marginal of $\pi$ is $\mu,$%
\[
\nabla\psi_{1}\left(  x\right)  =\nabla\psi_{2}\left(  x\right)  ,\text{ }%
\mu\text{-a.e. }x.
\]
To pass from $\mu$-a.e. to $\mathcal{L}^{d}$-a.e., we note that the set
$\operatorname*{int}\left(  \operatorname*{supp}\left(  \mu\right)  \right)  $
has full $\mu$-mass and $\rho>0$ $\mathcal{L}^{d}$-a.e. on it$.$ Thus the
equation translates to%
\[
\nabla\psi_{1}\left(  x\right)  =\nabla\psi_{2}\left(  x\right)  ,\text{
}\mathcal{L}^{d}\text{-a.e.}%
\]
The function $\psi_{1}-\psi_{2}$ is locally Lipschitz, hence belongs to
$W_{loc}^{1,\infty}\left(  \mathbb{R}^{d}\right)  $. It's gradient equals the
weak derivative $\mathcal{L}^{d}$-a.e. (See \cite[Section 5.8.3 Theorem
5]{evans2010partial}). A locally Lipschitz function with zero weak gradient on
a connected open set is constant. Therefore for some constant $C$,%
\[
\psi_{1}=\psi_{2}+C,\text{ }\mathcal{L}^{d}\text{-a.e. on }\operatorname*{int}%
\left(  \operatorname*{supp}\left(  \mu\right)  \right)  ,
\]
hence $\mu$-a.e. From $\left(  \ref{lm:uniq1_eq1}\right)  ,$%
\[
\varphi_{i}\left(  y\right)  =\psi_{i}\left(  x\right)  -c\left(  x,y\right)
,\text{ }\pi\text{-a.e. }\left(  x,y\right)  .
\]
Since the $y$-marginal of $\pi$ is $\nu,$ by Fubini theorem%
\[
\varphi_{1}\left(  y\right)  =\varphi_{2}\left(  y\right)  +C,\text{ }%
\nu\text{-a.e. }y.
\]

\textbf{2}. In the other case, assume $\nu$ satisfies (\textbf{Spt-a).} Assume
as before that $\left(  \psi_{i},\varphi_{i}\right)  $ are optmal potential
pair for $D_{c}\left(  \mu,P_{\leqslant\nu}\right)  $. The complementary
slackness $\left(  \ref{lm:uniq1_eq1}\right)  $ continues to hold. But the
backward-foreward swap tells us that $\left(  \psi_{i},Q_{\bar{c}}\left(
\psi_{i}\right)  \right)  $ is an optimal dual potential pair for
$D_{c}\left(  P_{\mu\leqslant},\nu\right)  ,$ therefore by complementary
slackness%
\[
\psi_{i}\left(  x\right)  -Q_{\bar{c}}\left(  \psi_{i}\right)  \left(
y\right)  =c\left(  x,y\right)  ,\text{ }\bar{\pi}\text{-a.e. }\left(
x,y\right)  ,
\]
where $\bar{\pi}$ is an optimal coupling between $\nu$ and its optimal
projection for $T_{c}\left(  P_{\mu\leqslant},\nu\right)  .$ The same line of
reasoning as the previous step shows that for some constant $C,$%
\[
Q_{\bar{c}}\left(  \psi_{1}\right)  \left(  y\right)  =Q_{\bar{c}}\left(
\psi_{2}\right)  \left(  y\right)  +C,\text{ }\nu\text{-a.e. }y.
\]
But Theorem \ref{thm:conj_opt_pair} says $Q_{\bar{c}}\left(  \psi_{i}\right)
=\varphi_{i},$ $\nu$-a.e. Thus%
\[
\varphi_{1}\left(  y\right)  =\varphi_{2}\left(  y\right)  +C,\text{ }%
\nu\text{-a.e. }y.
\]
It follows that%
\[
\psi_{1}\left(  x\right)  =\psi_{2}\left(  x\right)  +C,\text{ }\mu\text{-a.e.
}x.
\]

\textbf{3}. The uniqueness for forward dual follows from that for the backward
dual and Theorem \ref{thm:conj_opt_pair}.
\end{proof}

The condition (\textbf{Spt-a)} may be strong for some practical cases. The
next uniqueness asks for a weaker condition, it does not require absolute
continuity of source or target measure, but it needs the support of the source
measure rectifiably-connected, which is slightly stronger than being connected.

\begin{lemma}
\label{lm:uniq2}Suppose that $\mu$ satisfies (\textbf{Spt-b)},
$\operatorname*{supp}\left(  \nu\right)  $ is bounded and the cost $c\left(
x,y\right)  $ satisfies either one of the following.

$\left(  i\right)  $ $c\left(  x,y\right)  =h\left(  x-y\right)  $ with $h$
being differentiable, $\omega$-uniformly convex and satisfying the growth of
$\left(  \ref{def:growth_eq1}\right)  ,$

$\left(  ii\right)  $ $c\left(  x,y\right)  =\eta\left(  \left\vert
x-y\right\vert \right)  $ with $\eta\left(  s\right)  $ being differentiable
at $s>0$, $\eta_{+}^{\prime}\left(  0\right)  =0,$ $\omega$-uniformly convex
and satisfing the growth $\left(  \ref{def:growth_eq1}\right)  .$

Then (up to a constant) $D_{c}\left(  \mu,P_{\leqslant\nu}\right)  $ has a
unique dual potential pair. In particular, $D_{p}\left(  \mu,P_{\leqslant\nu
}\right)  $ with $p\geqslant2$ has a unique dual potential pair. The
uniqueness for forward dual $D_{c}\left(  P_{\mu\leqslant},\nu\right)  $ also
holds under the same assumption.
\end{lemma}

\begin{proof}
\textbf{1}. First consider item $\left(  i\right)  .$ The existence is
provided by Theorem \ref{thm:dual_att}. Since $\omega$-uniformly convexity
ensures the uniqueness of the optimal backward projection measure $\bar{\mu}$,
which has bounded support since $\operatorname*{supp}\left(  \nu\right)  $ is
bounded, so the uniqueness of the optimal potential follows from the proof of
\cite[Theorem 1.2]{ford2026quantitative} for differentiable cost.

\textbf{2}. The proof of $\left(  ii\right)  $ follows the same lines as the
case of differentiable cost except that we need to deal with the
non-differentiability of the cost on the diagonal. Let $\pi$ an optimal
coupling between $\mu$ and its optimal projection for $T_{c}\left(
\mu,P_{\leqslant\nu}\right)  .$ Pick $x_{0},$ $x_{1}\in\operatorname*{supp}%
\left(  \mu\right)  $ and a rectifiable curve $\gamma\left(  t\right)
:\left[  0,1\right]  \mapsto\operatorname*{supp}\left(  \mu\right)  $
connecting the two points (the curve is reparametrized by arc-length hence is
Lipschitz), $\gamma\left(  0\right)  =x_{0},$ $\gamma\left(  1\right)
=x_{1}.$ Suppose that both $\varphi_{i}\in\mathcal{A}\cap L^{1}\left(
d\nu\right)  ,$ $\psi_{i}=Q_{c}\left(  \varphi_{i}\right)  $ are optimal dual
and transport potentials, $i=1,2.$ Let $u_{i}=\psi_{i}\circ\gamma.$ Since
$\operatorname*{supp}\left(  \nu\right)  $ is bounded hence compact, every
$\gamma\left(  t\right)  $ corresponds to some $y\in\operatorname*{supp}%
\left(  \gamma\right)  .$ Using the compactness of $\operatorname*{supp}%
\left(  \nu\right)  $ again, we may write
\[
\psi_{i}\left(  x\right)  =\inf_{y^{\prime}\in\mathbb{R}^{d}}\left\{
\varphi_{i}\left(  y^{\prime}\right)  +c\left(  x,y^{\prime}\right)  \right\}
=\inf_{y^{\prime}\in\operatorname*{supp}\left(  \nu\right)  }\left\{
\varphi_{i}\left(  y^{\prime}\right)  +c\left(  x,y^{\prime}\right)  \right\}
,\text{ }\forall x.
\]
Hence $\psi_{i}$ is locally Lipschitz, thus the composition $u_{i}$ is locally
Lipschitz in $\left(  0,1\right)  ,$ so it is differentiable for almost all
$t\in\left(  0,1\right)  .$ Fix $t\in\left(  0,1\right)  $ and $\delta$ small
so that $t+\delta\in\left(  0,1\right)  .$ By optimality
\[
u_{i}\left(  t+\delta\right)  -u_{i}\left(  t\right)  \leqslant c\left(
\gamma\left(  t+\delta\right)  ,y\right)  -c\left(  \gamma\left(  t\right)
,y\right)  .
\]
Dividing by $\delta$ and taking $\delta\rightarrow0+$,%
\[
u_{i}^{\prime}\left(  t\right)  \leqslant\max\left\{  q\cdot\dot{\gamma
}\left(  t\right)  :q\in\partial_{x}c\left(  \gamma\left(  t\right)
,y\right)  \right\}  .
\]
Taking $\delta\rightarrow0-$ yields%
\[
u_{i}^{\prime}\left(  t\right)  \geqslant\min\left\{  q\cdot\dot{\gamma
}\left(  t\right)  :q\in\partial_{x}c\left(  \gamma\left(  t\right)
,y\right)  \right\}  .
\]
Note the subdifferential $\partial_{x}c$ is a closed set, it is not generally
compact. But by assumption $\partial\eta\left(  \left\vert z\right\vert
\right)  =\nabla\eta\left(  \left\vert z\right\vert \right)  $ if $z\neq0$ and%
\[
\partial\eta\left(  \left\vert z\right\vert \right)  =\eta_{+}^{\prime}\left(
0\right)  \left\{  \omega\in\mathbb{R}^{d}:\left\vert \omega\right\vert
\leqslant1\right\}  =0,\text{ if }z=0.
\]
So $\partial\eta\left(  \left\vert z\right\vert \right)  $ contains only one
vector regardless of whether $z\neq0$ or not. Therefore $\partial_{x}c\left(
\gamma\left(  t\right)  ,y\right)  $ is always a singleton, thus the use of
$\max$/$\min$ (rather than $\sup$/$\inf$) is justified and the lower and upper
bound of $u_{i}^{\prime}\left(  t\right)  $ coincides, for all $t\in\left(
0,1\right)  ,$%
\[
u_{i}^{\prime}\left(  t\right)  =\left\{
\begin{array}
[c]{ll}%
\nabla\eta\left(  \left\vert \gamma\left(  t\right)  -y\right\vert \right)
\cdot\dot{\gamma}\left(  t\right)  , & \text{if }\gamma\left(  t\right)  \neq
y,\\
0, & \text{if }\gamma\left(  t\right)  =y.
\end{array}
\right.
\]
This shows that $u_{i}^{\prime}\left(  t\right)  $ is uniquely determined by
$\eta,$ $\pi$, $\gamma$, independently of the optimal transport potential
$\psi_{i}.$ So we conclude by the fundamental theorem for Lebesgue integration
that the optimal dual and transport potentials are unique up to a constant.
\end{proof}

\section{Stability of the Wasserstein projection distance}

With the aid of backward-forward swap\textbf{ (}Theorem \ref{thm:bf_swap}), we
have the Lipschitz stability of the Wasserstein projection distances following
the argument of \cite{kim2024statistical} (see also \cite[Proposition 3.1,
4.3]{alfonsi2020sampling}). Recall that the Wassestein projection with cost
$\left\vert x-y\right\vert ^{p}$ are written $T_{p}\left(  \mu,P_{\leqslant
\nu}\right)  ,$ $T_{p}\left(  P_{\mu\leqslant},\nu\right)  $, and the
Wassestein projection distances are%
\begin{equation}
W_{p}\left(  \mu,P_{\leqslant\nu}\right)  =\left[  T_{p}\left(  \mu
,P_{\leqslant\nu}\right)  \right]  ^{1/p},\text{ }W_{p}\left(  P_{\mu
\leqslant},\nu\right)  =\left[  T_{p}\left(  P_{\mu\leqslant},\nu\right)
\right]  ^{1/p}. \label{eq:p_proj_dist}%
\end{equation}

\begin{theorem}
[\textbf{Lipschitz stability}]\label{thm:stability_dist}Let $p\geqslant1$,
$\mu,\nu,\mu^{\prime},\nu^{\prime}\in P_{p}\left(  \mathbb{R}^{d}\right)  $,
$c\left(  x,y\right)  =\left\vert x-y\right\vert ^{p}$. Then%
\[
\left\vert W_{p}\left(  \mu,P_{\leqslant\nu}\right)  -W_{p}\left(  \mu
^{\prime},P_{\leqslant\nu^{\prime}}\right)  \right\vert \leqslant W_{p}\left(
\mu,\mu^{\prime}\right)  +W_{p}\left(  \nu,\nu^{\prime}\right)  .
\]
and%
\[
\left\vert W_{p}\left(  P_{\mu\leqslant},\nu\right)  -W_{p}\left(
P_{\mu^{\prime}\leqslant},\nu^{\prime}\right)  \right\vert \leqslant
W_{p}\left(  \mu,\mu^{\prime}\right)  +W_{p}\left(  \nu,\nu^{\prime}\right)
.
\]

\end{theorem}

\begin{remark}
Theorem \ref{thm:stability_dist} continues to hold with general convex cost,
for example, of the form $\eta\left(  \left\vert x-y\right\vert \right)  $
with $\eta$ being an Orlicz function satisfing appropriate growth $\left(
\ref{def:growth_eq1}\right)  $, but we need to define the Orlicz style
Wasserstein distance%
\[
W_{c}\left(  \mu,\nu\right)  =\inf\left\{  s\in\left(  0,\infty\right)
:T_{\eta\left(  \left\vert \cdot\right\vert /s\right)  }\left(  \mu
,\nu\right)  \leqslant1\right\}  ,
\]
where $T_{\eta\left(  \left\vert \cdot\right\vert /s\right)  }\left(  \mu
,\nu\right)  $ is the optimal transport with cost $\eta\left(  \left\vert
x-y\right\vert /s\right)  .$ It is readily verifiable that $W_{c}$ satisfies
the triangle inequality thus the definition is justified. The Orlicz style
Wasserstein projection distance is defined as%
\[
W_{c}\left(  P_{\mu\leqslant},\nu\right)  =\inf\left\{  s\in\left(
0,\infty\right)  :T_{\eta\left(  \left\vert \cdot\right\vert /s\right)
}\left(  P_{\mu\leqslant},\nu\right)  \leqslant1\right\}  .
\]
In the same fashion, $W_{c}\left(  \mu,P_{\leqslant\nu}\right)  $ is defined
and in view of Theorem \ref{thm:bf_swap},%
\[
W_{c}\left(  \mu,P_{\leqslant\nu}\right)  =W_{c}\left(  P_{\mu\leqslant}%
,\nu\right)  .
\]

\end{remark}

\section{Stability of the optimal dual potential}

A key property that is essential to the stability of optimal dual potential
involves a feature which we call half-space property. It allows us to exrtact
convergent subsequence from an almost arbitrary sequence of convex functions.
Recall that a sequence $f_{n}$ is \textit{locally uniformly bounded} in a set
$S$ if for every $x\in S,$ there is an (open) neighbourhood $U$ of $x$ such
that $f_{n}$ is uniformly bounded in $S\cap U.$ \textit{Locally uniform
convergence} of $f_{n}$ is defined similarly.

\begin{lemma}
[\textbf{Half-space property}]\label{lm:half_spc}Let $f_{n}$ be a sequence of
l.s.c. proper convex functions on $\mathbb{R}^{d}.$ There exists $z_{0}%
\in\mathbb{R}^{d}$ that satisfies the asymptotic boundedness:%
\[
\exists\text{\thinspace}z_{n}\rightarrow z_{0}\in\mathbb{R}^{d}\text{ such
that }f_{n}\left(  z_{n}\right)  \text{ is bounded.}%
\]
Let%
\[
U=\left\{  x:\sup_{n}f_{n}\left(  x\right)  <\infty\right\}  .
\]
Then

(i) the set $U$ is convex, $\bar{U}$ is non-empty and $f_{n}$ is locally
uniformly upper bounded in $\operatorname*{int}\left(  U\right)  $, if
moreover $\operatorname*{int}\left(  U\right)  $ contains a point that
satisfies the asymptotic boundedness or simply $z_{0}\in\operatorname*{int}%
\left(  U\right)  ,$ then $f_{n}$ is locally uniformly bounded in
$\operatorname*{int}\left(  U\right)  ;$

(ii) up to a subsequence $f_{n}\rightarrow\infty$ on every compact set $K$ in
$\left(  \bar{U}\right)  ^{c}$.
\end{lemma}

\begin{proof}
\textbf{1}. It sufficies to show that at a point, say $x_{0}$, where locally
uniform upper boundeness fails, there is a half-space passing through $x_{0}$
in which $f_{n}$ is pointwise unbounded. Suppose up to a subsequence that
$x_{n}\rightarrow x_{0}$ and $f_{n}\left(  x_{n}\right)  >n$ as $n\rightarrow
\infty.$ Since $f_{n}$ is l.s.c. proper, it is the supremum of a family of
supporting affine functions, hence there exists a pair $y_{n}\in\mathbb{R}%
^{d},$ $\tau_{n}\in\mathbb{R}$ such that%
\[
l_{n}\left(  x\right)  \triangleq x\cdot y_{n}+\tau_{n}\leqslant f_{n}\left(
x\right)  ,\text{\thinspace}\forall x\in\mathbb{R}^{d}%
\]
and%
\[
l_{n}\left(  x_{n}\right)  \geqslant n.
\]
For any $x^{\prime}\in\mathbb{R}^{d},$%
\begin{align}
f_{n}\left(  x^{\prime}\right)  -n  &  \geqslant f_{n}\left(  x^{\prime
}\right)  -l_{n}\left(  x_{n}\right) \label{lm:half_spc_eq1}\\
&  \geqslant x^{\prime}\cdot y_{n}+\tau_{n}-\left(  x_{n}\cdot y_{n}+\tau
_{n}\right) \nonumber\\
&  =\left(  x^{\prime}-x_{n}\right)  \cdot y_{n}.\nonumber
\end{align}
Evaluting at $x^{\prime}=z_{n}$,
\[
f_{n}\left(  z_{n}\right)  -n\geqslant\left(  z_{n}-x_{n}\right)  \cdot
y_{n}.
\]
As $n\rightarrow\infty,$ the RHS must diverges to $-\infty.$ This is possible
only if $\left\vert y_{n}\right\vert \rightarrow\infty.$ Up to a subsequence,
we may suppose that the unit directions $u_{n}$ of $y_{n}-x_{n}$ converge to
some $u_{0}$ on the unit ball. Define the half-space%
\begin{equation}
H=\left\{  x:\left(  x-x_{0}\right)  \cdot u_{0}>0\right\}  .
\label{lm:half_spc_eq2}%
\end{equation}
Fix $x\in H$. For all large $n,$
\begin{equation}
\left(  x-x_{n}\right)  \cdot\left(  y_{n}-x_{n}\right)  >0,
\label{lm:half_spc_eq3}%
\end{equation}
which gives%
\[
\left(  x-x_{n}\right)  \cdot y_{n}\geqslant\left(  x-x_{n}\right)  \cdot
x_{n}%
\]
Evaluting $\left(  \ref{lm:half_spc_eq1}\right)  $ at $x^{\prime}=x$ and using
the above inequality yield%
\[
f_{n}\left(  x\right)  -n\geqslant\left(  x-x_{n}\right)  \cdot y_{n}%
\geqslant\left(  x-x_{n}\right)  \cdot x_{n},
\]
Since the rightmost term is bounded, we must have $f_{n}\left(  x\right)
\rightarrow\infty$ as $n\rightarrow\infty.$ Since this holds for an arbitrary
$x\in H,$ the conclusion follows.

\textbf{2}. For any open ball $B\subset\subset H,$ there is a positive
constant $\delta>0,$ so that inequality $\left(  \ref{lm:half_spc_eq2}\right)
$ becomes strict on $\bar{B},$%
\[
\left(  x-x_{0}\right)  \cdot u_{0}>2\delta,\text{ }\forall x\in\bar{B}.
\]
Still on the subsequence where $y_{n}-x_{n}$ converge $u_{0},$ provided $n$ is
larger than some $n_{0}$ (independent of $x\in B$), we have%
\[
\left(  x-x_{n}\right)  \cdot\left(  y_{n}-x_{n}\right)  >\delta\left\vert
y_{n}-x_{n}\right\vert ,\text{ }\forall x\in\bar{B},
\]
Following the same line of reasoning as before, we find that%
\[
f_{n}\left(  x\right)  -n\geqslant\left(  x-x_{n}\right)  \cdot x_{n}%
+\delta\left\vert y_{n}-x_{n}\right\vert ,\text{ }\forall x\in\bar{B}.
\]
Since $\left(  x-x_{n}\right)  \cdot x_{n}$ is uniformly bounded below on
$\bar{B},$ we conclude that $f_{n}\left(  x\right)  \rightarrow\infty$
uniformly on $\bar{B}.$

\textbf{3}. The sought-after convex set $U$ is just the intersection of all
such half spaces $H^{c}.$ Since $f_{n}$ is bounded along the convergent
sequence $z_{n},$ $\bar{U}$ must be non-empty, in fact $z_{0}$ must stay in
the closure $\bar{U},$ for otherwise $z_{0}$ will be contained in one of the
open half space $H,$ yielding a contradiction. Finally if a point of the
asymptotic boundedness lies in $\operatorname*{int}\left(  U\right)  ,$ then
$f_{n}$ is bounded at the point, hence the locally uniform boundedness in
$\operatorname*{int}\left(  U\right)  $ follows. The proof is therefore completed.
\end{proof}

\begin{lemma}
[\textbf{Unbounded subgradients}]\label{lm:subg_unbded}Let $f_{n}$ be a
sequence of l.s.c. proper convex functions on $\mathbb{R}^{d}.$ Suppose that
there exist $z_{n}\rightarrow z_{0}\in\mathbb{R}^{d}$ so that $f_{n}\left(
z_{n}\right)  $ is bounded. If $K$ is a compact set such that $z_{0}\notin K$
and $f_{n}\rightarrow\infty$ uniformly on $K$, then their subgradients on $K$
diverge uniformly,%
\[
\lim_{n\rightarrow\infty}\inf_{x\in K}\inf\left\{  \left\vert y\right\vert
:y\in\partial f_{n}\left(  x\right)  \right\}  =\infty.
\]
We adopt the convention that $\inf\varnothing=\infty.$
\end{lemma}

\begin{proof}
Write%
\[
\left\vert \partial f_{n}\left(  x\right)  \right\vert =\left\{  \left\vert
y\right\vert :y\in\partial f_{n}\left(  x\right)  \right\}  \text{ and
}\left\vert \partial f_{n}\left(  K\right)  \right\vert =\cup_{x\in
K}\left\vert \partial f_{n}\left(  x\right)  \right\vert .
\]
If $\left\vert \partial f_{n}\left(  K\right)  \right\vert $ is empty$,$ then
we can exclude such $n$ from the calculation of the limit by the convention.
Consider any $n$ for which $\left\vert \partial f_{n}\left(  K\right)
\right\vert $ is non-empty. For any $x\in K$ such that $\partial f_{n}\left(
x\right)  $ is not empty, let $y_{n}\in\partial f_{n}\left(  x\right)  ,$
\[
f_{n}\left(  z_{n}\right)  -f_{n}\left(  x\right)  \geqslant\left(
z_{n}-x\right)  \cdot y_{n}\geqslant-\left\vert z_{n}-x\right\vert \left\vert
y_{n}\right\vert .
\]
By assumption $z_{n}$ must stay outside $K$ for all large $n.$ Since $K$ is
compact, we can assume that for some constants $\delta_{0},$ $\delta_{1}>0$
and for all large $n,$
\[
\delta_{0}<\left\vert z_{n}-x\right\vert <\delta_{1},\text{ }\forall x\in K.
\]
Therefore dividing the first inequality through by $\left\vert z_{n}%
-x\right\vert $ yields
\[
\left\vert y_{n}\right\vert \geqslant\frac{f_{n}\left(  x\right)
-f_{n}\left(  z_{n}\right)  }{\left\vert z_{n}-x\right\vert }\geqslant\frac
{1}{\delta_{1}}\left[  f_{n}\left(  x\right)  -f_{n}\left(  z_{n}\right)
\right]  .
\]
By assumption, there exists $N$ such that $f_{n}\left(  x\right)
-f_{n}\left(  z_{n}\right)  >0$ for $n>N$ and $x\in K$. Therefore%
\[
\inf\left\vert \partial f_{n}\left(  x\right)  \right\vert \geqslant
\frac{f_{n}\left(  x\right)  -f_{n}\left(  z_{n}\right)  }{\left\vert
z_{n}-x\right\vert }\geqslant\frac{1}{\delta_{1}}\left[  f_{n}\left(
x\right)  -f_{n}\left(  z_{n}\right)  \right]  .
\]
Hence%
\[
\inf_{x\in K}\inf\left\vert \partial f_{n}\left(  x\right)  \right\vert
\geqslant\frac{1}{\delta_{1}}\left[  \inf_{x\in K}f_{n}\left(  x\right)
-f_{n}\left(  z_{n}\right)  \right]  .
\]
But $f_{n}$ diverges to $\infty$ uniformly on $K,$ thus the desired conclusion follows.
\end{proof}

With the above preparations ready, we are in position to prove the stability
of the optimal dual potentials. Note that the locally uniform convergence of
the backward optimal dual potential $\varphi_{n}$ does not just happen in the
interior of the support $\operatorname*{supp}\left(  \nu\right)  $, but also
it includes part of the boundary of the support.

\begin{theorem}
[\textbf{Stability of dual potentials}]\label{thm:stb_potential}Let
$p\geqslant1$, $\mu,$ $\nu\in P_{p}\left(  \mathbb{R}^{d}\right)  .$ Assume
that $\nu$ is not supported on a proper affine hyperplane, and $\mu_{n}%
,\nu_{n}\in P_{p}\left(  \mathbb{R}^{d}\right)  $ converges weakly to $\mu
,\nu$. Let $\varphi_{n}$ be the optimal dual potential for $D_{c}\left(
\mu_{n},P_{\leqslant\nu_{n}}\right)  $ (obtained in Theorem \ref{thm:dual_att}%
) with $\varphi_{n}\left(  \bar{\nu}_{n}\right)  =0,$ $\bar{\nu}_{n}=\int
yd\nu_{n}$. If $p>1$ we assume additionally that $h$ is superlinear and%
\begin{equation}
\sup_{n}\int\left\vert x\right\vert ^{p}d\mu_{n},\text{ }\sup_{n}%
\int\left\vert y\right\vert ^{p}d\nu_{n}<\infty. \label{thm:stb_potential_eq1}%
\end{equation}
Then

$\left(  i\right)  $ there is a non-empty convex set $U$ such that
$\operatorname*{supp}\left(  \nu\right)  \subset\bar{U}$ and up to a
subsequence $\varphi_{n}$ converges locally uniformly in $\operatorname*{int}%
\left(  U\right)  ,$

$\left(  ii\right)  $ if the condition of uniqueness of Lemma \ref{lm:uniq1}
or Lemma \ref{lm:uniq2} is satisfied, then the whole sequence $\varphi_{n}$
converges locally uniformly to the unique optimal dual potential $\varphi$ of
$D_{c}\left(  \mu,P_{\leqslant\nu}\right)  $ on $U\cap\left(
\operatorname*{supp}\left(  \nu\right)  \right)  .$

$\left(  iii\right)  $ up to a subsequence the associated optimal transport
potential $\psi_{n}=Q_{c}\left(  \varphi_{n}\right)  $ converges locally
uniformly in $\mathbb{R}^{d}$, the whole sequence converges locally uniformly
in $\mathbb{R}^{d}$ if the uniqueness condition is satisfied.

Similar conclusion holds for forward projections $D_{c}\left(  P_{\mu
\leqslant},\nu\right)  .$
\end{theorem}

\begin{proof}
\textbf{1. }For $p=1$, the convergence is direct via Theorem
\ref{thm:dual_att_lip_cost}, hence we focus on $p>1.$ The first thing we need
is a bound on the sequence%
\[
m_{\nu_{n}}=\int yd\nu_{n}.
\]
Since $\nu_{n}$ has uniformly bounded $p$-th moment, by Valle\'{e}-Poussin
theorem $\nu_{n}$ is uniformly integrable. But $\nu_{n}$ converges weakly to
$\nu,$ hence%
\[
m_{\nu_{n}}=\int yd\nu_{n}\rightarrow m_{\nu}=\int yd\nu.
\]
\textbf{2}. By convexity%
\begin{equation}
h\left(  x-y\right)  \leqslant\frac{1}{2}h\left(  2x\right)  +\frac{1}%
{2}h\left(  -2y\right)  \leqslant C\left(  \left\vert x\right\vert
^{p}+\left\vert y\right\vert ^{p}\right)  \label{thm:stb_potential_eq2}%
\end{equation}
for some constant $C>0$ depending on $h.$ Using $T_{p}\left(  \mu
_{m},P_{\leqslant\nu}\right)  \leqslant T_{p}\left(  \mu_{m},\nu\right)  $
together with the assumption shows that the sequence $T_{c}\left(  \mu
_{n},P_{\leqslant\nu_{n}}\right)  $ is bounded$,$ hence the dual energies
$D_{c}\left(  \mu_{n},P_{\leqslant\nu_{n}}\right)  $ is also bounded$.$ By
Theorem \ref{thm:conj_opt_pair} we may assume that $\psi_{n}=Q_{c}\left(
\varphi_{n}\right)  ,$ $\varphi_{n}=Q_{\bar{c}}\left(  \psi_{n}\right)  $ and
Theorem \ref{thm:dual_att} ensures that, for $n$ greater than some $n_{0},$%
\[
-\kappa\left\vert x\right\vert -\alpha\leqslant\psi_{n}\left(  x\right)
\leqslant h\left(  x-m_{\nu_{n}}\right)  \leqslant h\left(  x-m_{\nu_{n_{0}}%
}\right)  +1,\text{ }\forall x,
\]
where $\kappa,$ $\alpha>0$ are constants depending only on the bound of
$m_{\nu_{n}},$ the cost and the bounds of the dual, hence independent of $n.$
Therefore we can always extract a subsequence of $\psi_{n}$ that locally
uniformly converges on $\mathbb{R}^{d}$.

\textbf{3}. Next we show that it is also possible to extract a subsequence of
$\varphi_{n}$ that locally uniformly converges on some subset of
$\mathbb{R}^{d}$ to be determined later. Note that by the choice, $\varphi
_{n}$ is the supremum of a family of l.s.c functions, hence l.s.c. Also
$\varphi_{n}\in\mathcal{A}\cap L^{1}\left(  d\nu\right)  $ means it must be
proper. By the half-space property (Lemma \ref{lm:half_spc}), there exist a
convex set $U$ and a subsequence of $\varphi_{n}$ so that $\varphi_{n}$ is
locally uniformly convergent in $\operatorname*{int}\left(  U\right)  $ and
$\varphi_{n}\rightarrow\infty$ locally uniformly outside $\bar{U}.$ With the
aid of this fact, we claim that%
\begin{equation}
\operatorname*{supp}\left(  \nu\right)  \subset\bar{U}.
\label{thm:stb_potential_eq2.1}%
\end{equation}
We will prove by contradiction. Suppose that the inclusion fails to hold, then
there exist an open ball $B$ centering at some point in $\operatorname*{supp}%
\left(  \nu\right)  $ such that $\bar{B}\subset\left(  \bar{U}\right)  ^{c}.$
By the definition of support, $\nu\left(  B\right)  >0.$ Hence by weak
convergence%
\[
\liminf_{n}\nu_{n}\left(  B\right)  \geqslant\nu\left(  B\right)  >0.
\]
Up to a subsequence, we can assume that%
\[
\nu_{n}\left(  B\right)  \geqslant\frac{1}{2}\nu\left(  B\right)  ,\text{
}\forall n.
\]
Since $\left(  \psi_{n},\varphi_{n}\right)  $ is an optimal dual potential
pair for the forward dual $D_{c}\left(  P_{\mu_{n}\leqslant},\nu_{n}\right)
,$ hence by complementary slackness%
\begin{equation}
\psi_{n}\left(  x\right)  -\varphi_{n}\left(  y\right)  =c\left(  x,y\right)
,\text{ }\bar{\pi}_{n}\text{-a.e. }\left(  x,y\right)  ,
\label{thm:stb_potential_eq3}%
\end{equation}
where $\bar{\pi}_{n}$ is an optimal coupling between $\nu_{n}$ and its optimal
projection meassure for $T_{c}\left(  P_{\mu_{n}\leqslant},\nu_{n}\right)  .$
This means for $\bar{\pi}_{n}$-almost every pair $\left(  x,y\right)  $,%
\[
y\in\arg\min_{y^{\prime}\in\mathbb{R}^{d}}\left[  \varphi_{n}\left(
y^{\prime}\right)  +c\left(  x,y^{\prime}\right)  \right]  .
\]
In particular, we consider any $\bar{\pi}_{n}$-a.e. pair $\left(  x,y\right)
$ with $y\in\bar{B}.$ Such pairs does exist, since $\nu_{n}\left(  \bar
{B}\right)  \geqslant\nu_{n}\left(  B\right)  >0$. Since $\varphi_{n}\left(
\cdot\right)  ,$ $c\left(  x,\cdot\right)  $ are convex, it is sufficient and
necessary that
\[
0\in\partial_{y}\left[  \varphi_{n}\left(  y\right)  +c\left(  x,y\right)
\right]  .
\]
where $\partial_{y}$ denote subgradient w.r.t. $y$-variable$.$ Since
$\varphi_{n}$ diverges to $\infty$ uniformly on $\bar{B},$ meanwhile
$\varphi_{n}\left(  m_{\nu_{n}}\right)  =0$ and $m_{\nu_{n}}\rightarrow
m_{\nu}$, thus $m_{\nu}$ must belong to $\bar{U}$, hence keeps a strictly
positive distance from $\bar{B}.$ Recalling that $\nu_{n}\left(  B\right)  >0$
and $\varphi_{n}\in L^{1}\left(  d\nu\right)  ,$ we see that there exists some
point (which can depend on $n$) in $B$ where $\varphi_{n}$ is finite,
therefore the line segment connecting this point to $m_{\nu_{n}}$ has
non-empty interior and is contained in $\operatorname*{dom}\left(  \varphi
_{n}\right)  $. Hence $\operatorname*{dom}\left(  \varphi_{n}\right)  $ has
non-empty relative interior. Meanwhile since $c\left(  x,\cdot\right)  $ is
finite everywhere, $\operatorname*{dom}\left(  c\left(  x,\cdot\right)
\right)  =\mathbb{R}^{d}.$ Therefore the relative interior of
$\operatorname*{dom}\left(  \varphi_{n}\right)  $ overlaps that of
$\operatorname*{dom}\left(  c\left(  x,\cdot\right)  \right)  $ nicely. It
follows that%
\[
\partial_{y}\left[  \varphi_{n}\left(  y\right)  +c\left(  x,y\right)
\right]  =\partial_{y}\varphi_{n}\left(  y\right)  +\partial_{y}c\left(
x,y\right)  .
\]
Thus $0\in\partial_{y}\varphi_{n}\left(  y\right)  +\partial_{y}c\left(
x,y\right)  ,$ so there exists $u_{n,y}$ such that $u_{n,y}\in\partial
\varphi_{n}\left(  y\right)  ,\,$ $-u_{n,y}\in\partial_{y}c\left(  x,y\right)
.$ By Lemma \ref{lm:subg_unbded}, the norm of the subgradient $\left\vert
u_{n,y}\right\vert $ explodes as $n\rightarrow\infty$ uniformly for $y\in
\bar{B}.$ On substitution $\partial_{y}c\left(  x,y\right)  =-\partial
h\left(  x-y\right)  $. Also note that $-u_{n,y}\in-\partial h\left(
x-y\right)  $ if and only if $x-y\in\partial h^{\ast}\left(  u_{n,y}\right)
$. [Under the strictly convex and superlinear (recall $p>1$) assumption of
$h,$ the subgradient $\partial h^{\ast}$ is indeed the true gradient $\nabla
h^{\ast}$]. Since $h$ is finite everywhere and superlinear, $h^{\ast}$ must be
superlinear and%
\[
\lim_{\left\vert u\right\vert \rightarrow\infty}\inf\left\{  \left\vert
u^{\ast}\right\vert :u^{\ast}\in\partial h^{\ast}\left(  u\right)  \right\}
=\infty.
\]
Thus%
\[
\lim_{n\rightarrow\infty}\inf\left\{  \left\vert u^{\ast}\right\vert :u^{\ast
}\in\partial h^{\ast}\left(  u_{n,y}\right)  ,\text{ }y\in\bar{B}\right\}
=\infty.
\]
It follows that, since $x-y\in\partial h^{\ast}\left(  u_{n,y}\right)  ,$
$\left\vert x-y\right\vert $ must diverge to $\infty$ uniformly for $y\in
\bar{B}$ as $n\rightarrow\infty.$ This implies the $\nu_{n}$-mass in $\bar
{B},$ which is at least $\frac{1}{2}\nu\left(  B\right)  ,$ is transported
under $\bar{\pi}_{n}$ to the location $x$ which approaches $\infty$ uniformly
for $y\in\bar{B}$ as $n\rightarrow\infty.$ Let%
\[
\bar{\pi}_{n}^{-1}\left(  \bar{B}\right)  =\left\{  x:\text{there is }y\in
\bar{B}\text{ such that }\left(  \ref{thm:stb_potential_eq3}\right)  \text{
holds for }\left(  x,y\right)  \right\}  .
\]
What we have just shown is: for any compact $K\subset\mathbb{R}^{d}$, there is
$n_{0}$ such that whenever $n>n_{0},$%
\[
\bar{\pi}_{n}^{-1}\left(  \bar{B}\right)  \in K^{c}.
\]
Since the first marginal of $\bar{\pi}_{n}$ is $\mu_{n},$%
\[
\mu_{n}\left(  K^{c}\right)  \geqslant\mu_{n}\left[  \bar{\pi}_{n}^{-1}\left(
\bar{B}\right)  \right]  =\bar{\pi}_{n}\left[  \bar{\pi}_{n}^{-1}\left(
\bar{B}\right)  \times\mathbb{R}^{d}\right]  \geqslant\bar{\pi}_{n}\left[
\bar{\pi}_{n}^{-1}\left(  \bar{B}\right)  \times\bar{B}\right]  .
\]
By the definition of $\bar{\pi}_{n}^{-1}\left(  \bar{B}\right)  :$ $\bar{\pi
}_{n}\left[  \bar{\pi}_{n}^{-1}\left(  \bar{B}\right)  \times\bar{B}\right]
=\bar{\pi}_{n}\left[  \mathbb{R}^{d}\times\bar{B}\right]  =\nu_{n}\left(
\bar{B}\right)  .$ Hence%
\[
\mu_{n}\left(  K^{c}\right)  \geqslant\nu_{n}\left(  \bar{B}\right)  >\frac
{1}{2}\nu\left(  B\right)  ,\text{ }\forall n>n_{0}.
\]
But since $p>1,$ the assumption $\left(  \ref{thm:stb_potential_eq1}\right)  $
indicates that the sequence $\mu_{n}$ is tight. This is a contradiction, hence
$\left(  \ref{thm:stb_potential_eq2.1}\right)  $ is proved. Let $V_{0}$ be the
interior of the convex hull of $\operatorname*{supp}\left(  \nu\right)  .$ By
$\left(  \ref{thm:stb_potential_eq2.1}\right)  ,$ $V_{0}\subset
\operatorname*{int}\left(  U\right)  $. Since $\nu$ is not supported on a
proper affine hyperplane, the $V_{0}$ is non-empty and contains $m_{\nu}$.
Thus $m_{\nu}\in\operatorname*{int}\left(  U\right)  ,$ so Lemma
\ref{lm:half_spc} shows that up to a subsequence $\varphi_{n}$ converges
locally uniformly in $\operatorname*{int}\left(  U\right)  .$

\textbf{4}. The convergence of $\psi_{n}$ is proved in the same fashion as
$\varphi_{n},$ but is much simpler since $\psi_{n}$ is pointwise bounded
everywhere in $\mathbb{R}^{d}.$
\end{proof}

\begin{remark}
Since $m_{\nu}$ always lies in the relative interior of the cloure of the
convex hull of $\operatorname*{supp}\left(  \nu\right)  ,$ if $\nu$ is not
supported in a proper hyperplane, then the convergence of Theorem
\ref{thm:stb_potential} (ii) gives $\varphi\left(  m_{\nu}\right)  =0.$
\end{remark}

\section{Variance bounds of empirical Wasserstein projection}

In the following we will focus on strictly convex $p$-cost $h\left(
x-y\right)  =\left\vert x-y\right\vert ^{p}$, $p>1$. Extension to other
strictly convex costs (that satisfy growth condition $\left(
\ref{def:growth_eq1}\right)  $ for $p\geqslant1$) follows similar lines except
that some Orlicz space language might be needed. For $p$-cost, the $c/\bar{c}%
$-transform will be written as $Q_{p}\left(  \cdot\right)  ,$ $Q_{\bar{p}%
}\left(  \cdot\right)  .$

Let $X_{1},...,X_{m}$ be independent random variables, $X_{1}^{\prime
},...,X_{m}^{\prime}$ be respectively independent copies of $X_{1},...,X_{m}$,
and $S=f\left(  X_{1},...,X_{m}\right)  $ be a square integrable function of
$\left(  X_{1},...,X_{m}\right)  .$ Write for $i=1,...,m,$%
\[
S_{i}=f\left(  X_{1},...,X_{i-1},X_{i}^{\prime},X_{i+1},...,X_{m}\right)  .
\]
Then Efron-Stein inequality gives
\cite{boucheron2013concentration,efron1981jackknife}%
\begin{equation}
\operatorname*{Var}\left(  S\right)  \leqslant\frac{1}{2}\sum_{i=1}%
^{m}E\left(  S-S_{i}\right)  ^{2}=\sum_{i=1}^{m}E\left(  S-S_{i}\right)
_{+}^{2}, \label{ESinq}%
\end{equation}
where, for a real number $x\in\mathbb{R},$ $\left(  x\right)  _{+}%
=\max\left\{  x,0\right\}  $ are $\left(  x\right)  _{-}=\max\left\{
-x,0\right\}  $ are positive and negative part of $x.$

The initial step towards our central limit theorem is a bound on the variance
of the sequence of the empirical projection costs, which will readily give us
the tightness of the sequence.

Let $X_{1},...,X_{m}\sim\mu$ be independent random samples, and $X_{1}%
^{\prime},...,X_{m}^{\prime}$ be respectively independent copies of
$X_{1},...,X_{m}.$ Also let $Y_{1},...,Y_{n}\sim\nu$ be independent random
samples, and $Y_{1}^{\prime},...,Y_{n}^{\prime}$ be respectively independent
copies of $Y_{1},...,Y_{n}$. Consider the empirical measures%
\[
\mu_{m}=\frac{1}{m}\sum_{i=1}^{m}\delta_{X_{i}},\text{ }\nu_{n}=\frac{1}%
{n}\sum_{j=1}^{n}\delta_{Y_{i}}%
\]
and their symmetrized versions,%
\[
\mu_{m}^{\left(  i\right)  }=\frac{1}{m}\delta_{X_{i}^{\prime}}+\frac{1}%
{m}\sum_{l\neq i}^{m}\delta_{X_{l}},\text{ }\nu_{n}^{\left(  j\right)  }%
=\frac{1}{n}Y_{j}^{\prime}+\frac{1}{n}\sum_{l\neq j}^{n}\delta_{Y_{l}},\text{
}i=1,...,m,\text{ }j=1,...,n.
\]

\begin{theorem}
\label{thm:var_bd}Let $p>1,$ $\epsilon\geqslant0,$ $\mu,\nu\in P_{2p+\epsilon
}\left(  \mathbb{R}^{d}\right)  .$ Write%
\[
S_{m}=T_{p}\left(  \mu_{m},P_{\leqslant\nu}\right)  ,\text{ }S_{m}^{\left(
i\right)  }=T_{p}(\mu_{m}^{\left(  i\right)  },P_{\leqslant\nu}).
\]
Then
\[
m^{2+\epsilon/p}E[S_{m}-S_{m}^{(i)}]_{+}^{2+\epsilon/p}\leqslant C_{\mu,\nu}%
\]
where (up a constant depending on $p$ and $\epsilon$ only)%
\[
C_{\mu,\nu}=\int\left\vert x\right\vert ^{2p+\epsilon}d\mu+\int\left\vert
y\right\vert ^{2p+\epsilon}d\nu.
\]
Moreover
\[
\operatorname*{Var}\left(  T_{p}\left(  \mu_{m},P_{\leqslant\nu}\right)
\right)  \leqslant\frac{C_{\mu,\nu}}{m}.
\]

\end{theorem}

\begin{proof}
\textbf{1}. First we get an upper bound by the convexity of the $p$-cost,%
\begin{equation}
T_{p}\left(  \mu_{m},P_{\leqslant\nu}\right)  \leqslant T_{p}\left(  \mu
_{m},\nu\right)  \leqslant2^{p-1}\left(  \int\left\vert x\right\vert ^{p}%
d\mu_{m}+\int\left\vert y\right\vert ^{p}d\nu\right)  .
\label{thm:var_bd_inq1}%
\end{equation}
Let $\eta\left(  s\right)  =s^{p},$ $s\geqslant0,$ $p>1.$ By convexity,
\[
\eta\left(  s\right)  -\eta\left(  t\right)  \geqslant\eta^{\prime}\left(
t\right)  \left(  s-t\right)  ,\,\forall s,\,t\geqslant0.
\]
Switching the roles of $s$ and $t$ gives%
\[
\eta\left(  t\right)  -\eta\left(  s\right)  \geqslant\eta^{\prime}\left(
s\right)  \left(  t-s\right)  ,\,\forall s,\,t\geqslant0.
\]
Hence%
\begin{equation}
\left\vert \eta\left(  s\right)  -\eta\left(  t\right)  \right\vert
\leqslant\left\vert s-t\right\vert \left(  \left\vert \eta^{\prime}\left(
s\right)  \right\vert +\left\vert \eta^{\prime}\left(  t\right)  \right\vert
\right)  =p\left\vert s-t\right\vert \left(  \left\vert s\right\vert
^{p-1}+\left\vert t\right\vert ^{p-1}\right)  . \label{thm:var_bd_inq2}%
\end{equation}
For notational simplicity we will in the following omit constants depending
only on $p$ and $\varepsilon$. Combining inequality $\left(
\ref{thm:var_bd_inq1}\right)  \left(  \ref{thm:var_bd_inq2}\right)  $ and the
Lipschitz stability (Theorem \ref{thm:stability_dist}) gives
\begin{align*}
\lbrack S_{m}-S_{m}^{\left(  i\right)  }]_{+}  &  =\left\vert T_{p}\left(
\mu_{m},P_{\leqslant\nu}\right)  -T_{p}(\mu_{m}^{\left(  i\right)
},P_{\leqslant\nu})\right\vert \\
&  \leqslant\left\vert W_{p}\left(  \mu_{m},P_{\leqslant\nu}\right)
-W_{p}(\mu_{m}^{\left(  i\right)  },P_{\leqslant\nu})\right\vert \left(
\left[  W_{p}\left(  \mu_{m},P_{\leqslant\nu}\right)  \right]  ^{p-1}+\left[
W_{p}(\mu_{m}^{\left(  i\right)  },P_{\leqslant\nu})\right]  ^{p-1}\right) \\
&  \leqslant W_{p}(\mu_{m},\mu_{m}^{\left(  i\right)  })\left(  \left[
W_{p}\left(  \mu_{m},\nu\right)  \right]  ^{p-1}+\left[  W_{p}(\mu
_{m}^{\left(  i\right)  },\nu)\right]  ^{p-1}\right) \\
&  \leqslant\frac{1}{m}\left\vert X_{i}-X_{i}^{\prime}\right\vert \left[
\left(  \int\left\vert x\right\vert ^{p}d\mu_{m}+\int\left\vert y\right\vert
^{p}d\nu\right)  ^{\left(  p-1\right)  /p}+\left(  \int\left\vert x\right\vert
^{p}d\mu_{m}^{\left(  i\right)  }+\int\left\vert y\right\vert ^{p}d\nu\right)
^{\left(  p-1\right)  /p}\right]  ,
\end{align*}
where the last inequality follows by considering a coupling between $\mu
_{m},\mu_{m}^{\left(  i\right)  }$ such that $X_{1},...,X_{i-1},$
$X_{i+1},...,X_{m}$ stay and $X_{i}$ is transported to $X_{i}^{\prime}.$

\textbf{2}. Let $r=2+\epsilon/p.$ Take expectation of the $r$-th power of the
above inequality,%
\begin{align*}
&  E[S_{m}-S_{m}^{\left(  i\right)  }]_{+}^{r}\\
&  \leqslant\frac{1}{m^{r}}E\left\{  \left\vert X_{i}-X_{i}^{\prime
}\right\vert ^{r}\left[  \left(  \int\left\vert x\right\vert ^{p}d\mu_{m}%
+\int\left\vert y\right\vert ^{p}d\nu\right)  ^{\left(  p-1\right)
/p}+\left(  \int\left\vert x\right\vert ^{p}d\mu_{m}^{\left(  i\right)  }%
+\int\left\vert y\right\vert ^{p}d\nu\right)  ^{\left(  p-1\right)
/p}\right]  ^{r}\right\} \\
&  \leqslant\frac{1}{m^{r}}E\left\{  \left\vert X_{i}-X_{i}^{\prime
}\right\vert ^{r}\left[  \left(  \int\left\vert x\right\vert ^{p}d\mu_{m}%
+\int\left\vert y\right\vert ^{p}d\nu\right)  ^{r\left(  p-1\right)
/p}+\left(  \int\left\vert x\right\vert ^{p}d\mu_{m}^{\left(  i\right)  }%
+\int\left\vert y\right\vert ^{p}d\nu\right)  ^{r\left(  p-1\right)
/p}\right]  \right\}  .
\end{align*}
Let $q=p/\left(  p-1\right)  $ be the conjugate index of $p,$ $p^{-1}%
+q^{-1}=1.$ Now we estimate the first expectation (similarly the second
expectation) via H\"{o}lder's inequality and Jensen's inequality%
\begin{align*}
&  E\left(  \left\vert X_{i}-X_{i}^{\prime}\right\vert ^{r}\left(
\int\left\vert x\right\vert ^{p}d\mu_{m}+\int\left\vert y\right\vert ^{p}%
d\nu\right)  ^{r\left(  p-1\right)  /p}\right) \\
&  \leqslant\left[  E\left\vert X_{i}-X_{i}^{\prime}\right\vert ^{rp}\right]
^{1/p}\left[  E\left(  \int\left\vert x\right\vert ^{p}d\mu_{m}+\int\left\vert
y\right\vert ^{p}d\nu\right)  ^{rq\left(  p-1\right)  /p}\right]  ^{1/q}\\
&  \leqslant\left[  E\left(  \left\vert X_{i}\right\vert ^{rp}+\left\vert
X_{i}^{\prime}\right\vert ^{rp}\right)  \right]  ^{1/p}\left[  E\left(
\int\left\vert x\right\vert ^{p}d\mu_{m}\right)  ^{r}+E\left(  \int\left\vert
y\right\vert ^{p}d\nu\right)  ^{r}\right]  ^{1/q}\\
&  \leqslant\left[  E\left\vert X_{i}\right\vert ^{rp}+E\left\vert
X_{i}^{\prime}\right\vert ^{rp}\right]  ^{1/p}\left[  E\left(  \int\left\vert
x\right\vert ^{rp}d\mu_{m}\right)  +E\left(  \int\left\vert y\right\vert
^{rp}d\nu\right)  \right]  ^{1/q}.
\end{align*}
Combining these estimates, we proceed to get that%
\begin{align*}
E[S_{m}-S_{m}^{\left(  i\right)  }]_{+}^{r}  &  \leqslant\frac{1}{m^{r}%
}\left\{  \left(  \int\left\vert x\right\vert ^{rp}d\mu\right)  ^{1/p}\left(
\int\left\vert x\right\vert ^{rp}d\mu+\int\left\vert y\right\vert ^{rp}%
d\nu\right)  ^{1/q}\right\} \\
&  \leqslant\frac{1}{m^{r}}\left(  \int\left\vert x\right\vert ^{rp}d\mu
+\int\left\vert y\right\vert ^{rp}d\nu\right)  =\frac{C_{\mu,\nu}}{m^{r}}.
\end{align*}
Therefore if we set $\epsilon=0$ (i.e. $r=2$) and use Efron-Stein inequality%
\[
\operatorname*{Var}\left(  T_{p}\left(  \mu_{m},P_{\leqslant\nu}\right)
\right)  =\operatorname*{Var}\left(  S_{m}\right)  \leqslant\sum_{i=1}%
^{m}E[S_{m}-S_{m}^{\left(  i\right)  }]_{+}^{2}\leqslant\frac{C_{\mu,\nu}}%
{m}.
\]

\end{proof}

If the perturbation happens on the vertex side, write%
\[
S_{n}=T_{p}\left(  \mu,P_{\leqslant\nu_{n}}\right)  ,\text{ }S_{n}^{\left(
j\right)  }=T_{p}(\mu,P_{\leqslant\nu_{n}^{\left(  j\right)  }}).
\]
then Theorem \ref{thm:var_bd} holds by symmetry%
\begin{equation}
n^{2+\epsilon/p}E[S_{n}-S_{n}^{(j)}]_{+}^{2+\epsilon/p}\leqslant C_{\mu,\nu
}\text{.} \label{var_bd_Snj_eq}%
\end{equation}
Setting $\epsilon=0$ gives the variance bound%
\begin{equation}
\operatorname*{Var}\left(  T_{p}\left(  \mu,P_{\leqslant\nu_{n}}\right)
\right)  \leqslant\frac{C_{\mu,\nu}}{n}. \label{var_bd_Snj_eq1}%
\end{equation}
For the double-perturbation case, let%
\[
S_{m,n}=T_{p}\left(  \mu_{m},P_{\leqslant\nu_{n}}\right)  ,\text{ }%
S_{m,n}^{\left(  l\right)  }=\left\{
\begin{array}
[c]{ll}%
T_{p}(\mu_{m}^{\left(  l\right)  },P_{\leqslant\nu_{n}}), & \text{if
}l=1,...,m,\\
T_{p}(\mu_{m},P_{\leqslant\nu_{n}^{\left(  l-m\right)  }}), & \text{if
}l=m+1,...,m+n,
\end{array}
\right.
\]
then we have%
\[
E[S_{m,n}-S_{m,n}^{\left(  l\right)  }]_{+}^{2+\epsilon/p}\leqslant
\frac{C_{\mu,\nu}}{m^{2+\epsilon/p}}+\frac{C_{\mu,\nu}}{n^{2+\epsilon/p}%
}\text{,}%
\]%
\[
\operatorname*{Var}\left(  T_{p}\left(  \mu_{m},P_{\leqslant\nu_{n}}\right)
\right)  \leqslant E\sum_{l=1}^{m+n}[S_{m,n}-S_{m,n}^{\left(  l\right)  }%
]_{+}^{2}\leqslant\frac{C_{\mu,\nu}}{m}+\frac{C_{\mu,\nu}}{n}.
\]

\section{Central limit theorem}

\subsection{Linearization}

Before we start, we present some \textit{heuristics} that offer a glimpse into
what it takes to prove a central limit theorem for Wasserstein projection. A
major difficulty in obtaining a central limit theorem for Wasserstein
projection distance is the nonlinear dependence of the projection distance on
the marginals A generic approach is to linearize the main objective and
estimate the higher order term. The idea of linearization was implcitly used
with the proof of the central limit theorem for the classical Wasserstein
distance \cite{del2019central}, \cite{del2024central}. Here we will explain
explicitly how the idea of linearization works for Wasserstein projection. We
note again that one of the crucial ingredients is the Lipschitz stability of
Wasserstein projection obtained through the backward-forward swap property.

Let $\mu,\,\nu\in P_{p}\left(  \mathbb{R}^{d}\right)  .$ Write the projection
as a functional of $\mu,$ $\nu,$%
\[
T_{p}\left(  \mu,P_{\leqslant\nu}\right)  =\sup_{\varphi\in\mathcal{A}\cap
C_{b,p}}\xi\left(  \varphi,\mu,\nu\right)  \triangleq\int Q_{p}\left(
\varphi\right)  d\mu-\int\varphi d\nu.
\]
Let $\mu^{\prime}\in P_{p}\left(  \mathbb{R}^{d}\right)  $ be a perturbation
of $\mu.$ We are concerned with the effect of the perturbation on the
projection cost $T_{p}\left(  \mu,P_{\leqslant\nu}\right)  $. Denote the
unique optimal potential associated with the projeciton distance by
$\varphi_{\mu,\nu}$. Then%
\[
T_{p}\left(  \mu,P_{\leqslant\nu}\right)  =\xi\left(  \varphi_{\mu,\nu}%
,\mu,\nu\right)  .
\]
Now let us pretend that it is legal to differentiate in the above equation
with respect to $\mu$ and we denote the derivative with respect to $\mu$ by
$\delta\left(  \cdot\right)  /\delta\mu.$ The derivative w.r.t. a probability
measure is understood in the Wasserstein space. Then%
\[
\frac{\delta}{\delta\mu}T_{p}\left(  \mu,P_{\leqslant\nu}\right)
=\frac{\partial}{\partial\varphi}\xi\left(  \varphi_{\mu,\nu},\mu,\nu\right)
\cdot\frac{\delta\varphi_{\mu,\nu}}{\delta\mu}+\frac{\delta}{\delta\mu}%
\xi\left(  \varphi_{\mu,\nu},\mu,\nu\right)  =\frac{\delta}{\delta\mu}%
\xi\left(  \varphi_{\mu,\nu},\mu,\nu\right)  .
\]
Here $\delta\xi/\delta\varphi,$ $\delta\xi/\delta\mu$ denote the derivatives
of $\xi=\xi\left(  \varphi,\mu,\nu\right)  $ w.r.t. its first and second
argument. Since $\varphi_{\mu,\nu}$ is optimal for $\xi\left(  \cdot,\mu
,\nu\right)  $ when $\mu,\nu$ held fixed, $\partial\xi/\partial\varphi
_{\mu,\nu}=0$. It follows that%
\[
\frac{\delta}{\delta\mu}T_{p}\left(  \mu,P_{\leqslant\nu}\right)
=\frac{\delta}{\delta\mu}\xi\left(  \varphi_{\mu,\nu},\mu,\nu\right)
=Q_{p}\left(  \varphi_{\mu,\nu}\right)  .
\]
Hence we may write%
\begin{equation}
T_{p}\left(  \mu^{\prime},P_{\leqslant\nu}\right)  =T_{p}\left(
\mu,P_{\leqslant\nu}\right)  +\int Q_{p}\left(  \varphi_{\mu,\nu}\right)
d\left(  \mu^{\prime}-\mu\right)  +r_{p}\left(  \mu^{\prime};\mu
,P_{\leqslant\nu}\right)  \label{eq:linearize}%
\end{equation}
with $r_{p}\left(  \mu^{\prime};\mu,P_{\leqslant\nu}\right)  $ being the
higher order (\textit{nonlinear}) term such that%
\begin{equation}
\lim_{W_{p}\left(  \mu^{\prime},\mu\right)  \rightarrow0}r_{p}\left(
\mu^{\prime};\mu,P_{\leqslant\nu}\right)  =0. \label{eq:lin_r2}%
\end{equation}
Now replacing $\mu^{\prime}$ in $\left(  \ref{eq:linearize}\right)  $ by a
random empirical sample $\mu_{m}$ $\ $of $\mu$ results in the equation%
\begin{align*}
T_{p}\left(  \mu_{m},P_{\leqslant\nu}\right)   &  =T_{p}\left(  \mu
,P_{\leqslant\nu}\right)  +\int Q_{p}\left(  \varphi_{\mu,\nu}\right)
d\left(  \mu_{m}-\mu\right)  +r_{p}\left(  \mu_{m};\mu,P_{\leqslant\nu}\right)
\\
&  =\int Q_{p}\left(  \varphi_{\mu,\nu}\right)  d\mu_{m}+\left[  r_{p}\left(
\mu_{m};\mu,P_{\leqslant\nu}\right)  -\int\varphi_{\mu,\nu}d\nu\right]  .
\end{align*}
It becomes clear that, since $\left(  \ref{eq:lin_r2}\right)  $ roughly
ensures that $r_{p}\left(  \mu_{m};\mu,P_{\leqslant\nu}\right)  $ converges
weakly to zero as soon as $m\rightarrow\infty,$ a central limit theorem of
$T_{p}\left(  \mu_{m},P_{\leqslant\nu}\right)  $ is easily obtainable from
that of $\int Q_{p}\left(  \varphi_{\mu,\nu}\right)  d\mu_{m}$. These of
course are just heuristic derivation. However, despite the lack of rigor, it
does offer a glimpse of the central object that needs to be taken care of.
Rather than making the above argument rigorous, we can instead work directly
on what is produced by the heuristics, i.e., the nonlinear term%
\[
R_{m}=r_{p}\left(  \mu_{m};\mu,P_{\leqslant\nu}\right)  -\int\varphi_{\mu,\nu
}d\nu=T_{p}\left(  \mu_{m},P_{\leqslant\nu}\right)  -\int Q_{p}\left(
\varphi_{\mu,\nu}\right)  d\mu_{m}.
\]
Therefore the endeavour to prove a central limit theorem is to prove that the
scaled remainer converges weakly to zero,%
\[
\sqrt{m}\left[  R_{m}-E\left(  R_{m}\right)  \right]  \rightarrow0\text{
weakly.}%
\]
If the perturbation happens on the vertex $\nu$ or on both marginals$,$ the
same line of reasoning sugguests that we consider%
\[
R_{n}=T_{p}\left(  \mu,P_{\leqslant\nu_{n}}\right)  +\int\varphi_{\mu,\nu}%
d\nu_{n},
\]
or in the latter case%
\[
R_{m,n}=T_{p}\left(  \mu_{m},P_{\leqslant\nu_{n}}\right)  -\int Q_{p}\left(
\varphi_{\mu,\nu}\right)  d\mu_{m}+\int\varphi_{\mu,\nu}d\nu_{n}.
\]

Before we dive into the proof, a few new notations are in position. we write
$R_{m}^{\left(  i\right)  }$ for the version of $R_{m}$ where $\mu
_{m}^{\left(  i\right)  }$ is used in place of $\mu_{m}$, $R_{n}^{\left(
j\right)  }$ for the version of $R_{n}$ with $\nu_{n}^{\left(  j\right)  },$
and $R_{m,n}^{\left(  i,j\right)  }$ for the version of $R_{m,n}$ with
$\mu_{m}^{\left(  i\right)  },\nu_{n}^{\left(  j\right)  }.$

Since the optimal values of backward and forward projection are equal (Theorem
\ref{thm:bf_swap}), we will only state the central limit theorems for backward projection.

\subsection{The non-vertex perturbation case}

First we deal with sampling that happens at the marginal which is not the
vertex of the convex order cone.

\begin{theorem}
\label{thm:CLT}Let $p>1,$ $\mu\in P_{2p}\left(  \mathbb{R}^{d}\right)  ,$
$\nu\in P_{p}\left(  \mathbb{R}^{d}\right)  $ satisfy one of the uniqueness
conditions of Lemma \ref{lm:uniq1} or Lemma \ref{lm:uniq2}. Let $\varphi$ be
the unique optimal potential for $D_{p}\left(  \mu,P_{\leqslant\nu}\right)  $
with $\varphi\left(  m_{\nu}\right)  =0,$ $m_{\nu}=\int yd\nu$. Write
$\psi=Q_{p}\left(  \varphi\right)  .$ Then

$\left(  i\right)  $ $\sqrt{m}\left(  R_{m}-E\left(  R_{m}\right)  \right)  $
converges in $L^{2},$
\begin{equation}
m\operatorname*{Var}\left(  R_{m}\right)  \rightarrow0, \label{thm:CLT_eq1}%
\end{equation}

$\left(  ii\right)  $ the central limit theorem holds%
\[
\sqrt{m}\left[  T_{p}\left(  \mu_{m},P_{\leqslant\nu}\right)  -E\left(
T_{p}\left(  \mu_{m},P_{\leqslant\nu}\right)  \right)  \right]  \rightarrow
\mathcal{N}\left(  0,\sigma_{p,\mu}^{2}\right)  \text{ weakly},
\]
where%
\begin{equation}
\sigma_{p,\mu}^{2}=\int\psi^{2}d\mu-\left(  \int\psi d\mu\right)  ^{2}.
\label{thm:CLT_eq2}%
\end{equation}

\end{theorem}

\begin{proof}
\textbf{1}. Let $\varphi_{m}$ be an optimal potential for $D_{p}\left(
\mu_{m},P_{\leqslant\nu}\right)  $ . Write $\psi_{m}=Q_{p}\left(  \varphi
_{m}\right)  .$ First we form the difference%
\begin{align*}
R_{m}-R_{m}^{\left(  i\right)  }  &  =T_{p}\left(  \mu_{m},P_{\leqslant\nu
}\right)  -T_{p}\left(  \mu_{m}^{\left(  i\right)  },P_{\leqslant\nu}\right)
-\int\psi d\mu_{m}+\int\psi d\mu_{m}^{\left(  i\right)  }\\
&  \leqslant\int\psi_{m}\mu_{m}-\int\varphi_{m}d\nu-\left(  \int\psi_{m}%
d\mu_{m}^{\left(  i\right)  }-\int\varphi_{m}d\nu\right)  -\int\psi d\mu
_{m}+\int\psi d\mu_{m}^{\left(  i\right)  }\\
&  =\int\left(  \psi_{m}-\psi\right)  d\mu_{m}-\int\left(  \psi_{m}%
-\psi\right)  \mu_{m}^{\left(  i\right)  }\\
&  =\frac{1}{m}\left(  \psi_{m}-\psi\right)  \left(  X_{i}\right)  -\frac
{1}{m}\left(  \psi_{m}-\psi\right)  \left(  X_{i}^{\prime}\right)
\end{align*}
Hence%
\[
\lbrack R_{m}-R_{m}^{\left(  i\right)  }]_{+}\leqslant\frac{1}{m}\left\vert
\psi_{m}-\psi\right\vert \left(  X_{i}\right)  +\frac{1}{m}\left\vert \psi
_{m}-\psi\right\vert \left(  X_{i}^{\prime}\right)  .
\]
The same inequality holds with $X_{i}$ and $X_{i}^{\prime}$ exchanged$,$ hence
by the stability of optimal transport potentials (Theorem
\ref{thm:stb_potential})%
\[
m|R_{m}-R_{m}^{\left(  i\right)  }|\rightarrow0\text{ a.s.}%
\]
\ The conclusion follows from an application of Efron-Stein inequality
$\left(  \ref{ESinq}\right)  ,$%
\[
\operatorname*{Var}\left(  R_{m}\right)  \leqslant2\sum_{i=1}^{m}\frac
{1}{m^{2}}E\left\vert \psi_{m}-\psi\right\vert ^{2}\left(  X_{i}\right)
+\frac{1}{m^{2}}E\left\vert \psi_{m}-\psi\right\vert ^{2}\left(  X_{i}%
^{\prime}\right)  =\frac{4}{m}E\left\vert \psi_{m}-\psi\right\vert ^{2}\left(
X_{1}\right)  .
\]
Is follows that%
\[
m\operatorname*{Var}\left(  R_{m}\right)  \leqslant\int\left\vert \psi
_{m}-\psi\right\vert ^{2}d\mu.
\]
By Theorem \ref{thm:dual_att}, $\left\vert \psi_{m}-\psi\right\vert
\rightarrow0$ locally uniformly in $\mathbb{R}^{d}$ and $\left\vert \psi
_{m}-\psi\right\vert ^{2}$ has a dominant function $2h^{2}\left(  x-m_{\nu
}\right)  $ which is $\mu$-integrable, hence the conclusion follows from the
dominated convergence theorem.

\textbf{2}. By definition%
\[
T_{p}\left(  \mu_{m},P_{\leqslant\nu}\right)  -E\left[  T_{p}\left(  \mu
_{m},P_{\leqslant\nu}\right)  \right]  =R_{m}-E\left(  R_{m}\right)  +\int\psi
d\mu_{m}-E\left(  \int\psi d\mu_{m}\right)  .
\]
Note $\left(  \ref{thm:CLT_eq1}\right)  $ implies via Chebyshev's inequality
that%
\[
\sqrt{m}\left[  R_{m}-E\left(  R_{m}\right)  \right]  \rightarrow0\text{
weakly}.
\]
Plugging in the empirical measure, we have%
\[
\int\psi d\mu_{m}=\frac{1}{m}\sum_{i=1}^{m}\psi\left(  X_{i}\right)  .
\]
Since $\psi\left(  X_{1}\right)  ,...,\psi\left(  X_{m}\right)  $ are i.i.d
and%
\[
\operatorname*{Var}\left[  \psi\left(  X_{1}\right)  \right]  =\int\psi
^{2}d\mu-\left(  \int\psi d\mu\right)  ^{2}\leqslant\int h^{2}\left(
x-m_{\nu}\right)  d\mu<\infty,
\]
Chebyshev's central limit theorem yields that%
\[
\sqrt{m}\left[  T_{p}\left(  \mu_{m},P_{\leqslant\nu}\right)  -E\left[
T_{p}\left(  \mu_{m},P_{\leqslant\nu}\right)  \right]  \right]  \rightarrow
\mathcal{N}\left(  0,\sigma_{p,\mu}^{2}\right)  \text{ weakly}.
\]
The proof is completed by combining the above with Slutsky's lemma.
\end{proof}

\subsection{The vertex perturbation case}

For non-vertex marginal sampling of $D_{p}\left(  \mu,P_{\leqslant\nu}\right)
$, the variance convergence $\left(  \ref{thm:CLT_eq1}\right)  $ comes at no
cost of higher moments$:$ $m\operatorname*{Var}\left(  R_{m}\right)
\rightarrow0$. This together with the definition that $T_{p}\left(  \mu
_{m},P_{\leqslant\nu}\right)  =R_{m}+\int\psi d\mu_{m}$ yields the variance
convergence of $T_{p}\left(  \mu_{m},P_{\leqslant\nu}\right)  :$%
\[
m\operatorname*{Var}\left(  T_{p}\left(  \mu_{m},P_{\leqslant\nu}\right)
\right)  \rightarrow\sigma_{p,\mu}^{2}.
\]
But for vertex marginal sampling, the situation is different. To get the
analogous convergence, we need slightly stronger assumptions on the moments of
$\mu,$ $\nu.$

\begin{theorem}
\label{thm:CLT_vtx}Let $p>1,$ $\varepsilon>0,$ $\mu,$ $\nu\in
P_{2p+\varepsilon}\left(  \mathbb{R}^{d}\right)  $ satisfy one of the
uniqueness conditions of Lemma \ref{lm:uniq1} or Lemma \ref{lm:uniq2}. Let
$\varphi$ be the unique optimal potential for $D_{p}\left(  \mu,P_{\leqslant
\nu}\right)  $ with $m_{\nu}=\int yd\nu,$ $\varphi\left(  m_{\nu}\right)  =0$.
Write $\psi=Q_{p}\left(  \varphi\right)  .$ Then

$\left(  i\right)  $ $n\operatorname*{Var}\left(  R_{n}\right)  \rightarrow0,$

$\left(  ii\right)  $ $n\operatorname*{Var}\left(  T_{p}\left(  \mu
,P_{\leqslant\nu_{n}}\right)  \right)  \rightarrow0$ and the central limit
theorem holds%
\[
\sqrt{n}\left[  T_{p}\left(  \mu,P_{\leqslant\nu_{n}}\right)  -E\left(
T_{p}\left(  \mu,P_{\leqslant\nu_{n}}\right)  \right)  \right]  \rightarrow
\mathcal{N}\left(  0,\sigma_{p,\nu}^{2}\right)  \text{ weakly},
\]
where%
\begin{equation}
\sigma_{p,\nu}^{2}=\int\varphi^{2}d\nu-\left(  \int\varphi d\nu\right)  ^{2}.
\label{thm:CLT_vtx_eq0}%
\end{equation}

\end{theorem}

\begin{proof}
Recall that%
\[
R_{n}-R_{n}^{\left(  j\right)  }=T_{p}\left(  \mu,P_{\leqslant\nu_{n}}\right)
-T_{p}(\mu,P_{\leqslant\nu_{n}^{\left(  j\right)  }})+\int\varphi d\nu
_{n}-\int\varphi d\nu_{n}^{\left(  j\right)  }.
\]
Let $\varphi_{n}$ be an optimal potential for $D_{p}\left(  \mu,P_{\leqslant
\nu_{n}}\right)  $. Following the first part of the proof of Theorem
\ref{thm:CLT} we have%
\[
R_{n}-R_{n}^{\left(  j\right)  }\leqslant\frac{1}{n}\left(  \varphi
-\varphi_{n}\right)  (Y_{j})+\frac{1}{n}\left(  \varphi_{n}-\varphi\right)
(Y_{j}^{\prime}),
\]
and%
\begin{equation}
n|R_{n}-R_{n}^{\left(  j\right)  }|\rightarrow0\text{ a.s.}
\label{thm:CLT_vtx_eq1}%
\end{equation}
To prove $\left(  i\right)  ,$ it suffices to show that $n|R_{n}%
-R_{n}^{\left(  j\right)  }|\rightarrow0$ in $L^{2}$ and use Efron-Stein
inequality. Since we already have almost sure convergence $\left(
\ref{thm:CLT_vtx_eq1}\right)  ,$ the $L^{2}$ convergence will follow if
$n^{2}[R_{n}-R_{n}^{\left(  j\right)  }]^{2}$ is uniformly integrable. Since
$n^{2}[R_{n}-R_{n}^{\left(  j\right)  }]^{2}$ is dominated by the sum of
$n^{2}[T_{p}\left(  \mu,P_{\leqslant\nu_{n}}\right)  -T_{p}(\mu,P_{\leqslant
\nu_{n}^{\left(  j\right)  }})]^{2}$ and $n^{2}[\int\varphi d\nu_{n}%
-\int\varphi d\nu_{n}^{\left(  j\right)  }]^{2}$. It remains to show that the
dominated terms respectively uniformly integrable. On the one hand, since
$n[T_{p}\left(  \mu,P_{\leqslant\nu_{n}}\right)  -T_{p}(\mu,P_{\leqslant
\nu_{n}^{\left(  j\right)  }})]$ has finite $\left(  2+\varepsilon\right)
$-th moment (see $\left(  \ref{var_bd_Snj_eq}\right)  $), hence $n^{2}%
[T_{p}\left(  \mu,P_{\leqslant\nu_{n}}\right)  -T_{p}(\mu,P_{\leqslant\nu
_{n}^{\left(  j\right)  }})]^{2}$ is uniformly integrable. On the other hand,
let $\pi$ be an optimal coupling between $\mu$ and its optimal projection for
$T_{p}\left(  \mu,P_{\leqslant\nu}\right)  .$ By complementary slackness
\[
\psi\left(  x\right)  -\varphi\left(  y\right)  =\left\vert x-y\right\vert
^{p},\text{ }\pi\text{-a.e. }\left(  x,y\right)  .
\]
Integrating the equation against $\mu$ gives%
\[
\varphi\left(  y\right)  =\int\psi\left(  x\right)  d\mu\left(  x\right)
-\int\left\vert x-y\right\vert ^{p}d\mu\left(  x\right)  ,\text{ }%
\nu\text{-a.e. }y.
\]
Integrating the square of the equation against $\nu$ yields
\begin{align*}
\int\varphi^{2}d\nu &  \leqslant2\left(  \int\psi d\mu\right)  ^{2}%
+2\int\left(  \int\left\vert x-y\right\vert ^{p}d\mu\left(  x\right)  \right)
^{2}d\nu\left(  y\right) \\
&  \leqslant2\int\psi^{2}d\mu+2\int\left(  \int\left\vert x-y\right\vert
^{2p}d\mu\left(  x\right)  \right)  d\nu\left(  y\right)  <\infty.
\end{align*}
Therefore%
\[
n\left[  \int\varphi d\nu_{n}-\int\varphi d\nu_{n}^{\left(  j\right)
}\right]  =\varphi(Y_{j})-\varphi(Y_{j}^{\prime})
\]
is independent of $n$ and has finite second moment$,$ hence is trivially
uniform integrable. The proof of item $\left(  i\right)  $ is completed. The
remaining proof follows the same lines as Theorem \ref{thm:CLT}.
\end{proof}

\subsection{The double-perturbation case}

For sampling that happens at both the non-vertex and vertex marginal, the
result is a combination of Theorem \ref{thm:CLT} and Theorem \ref{thm:CLT_vtx}%
. We state without proof the central limit theorem below.

\begin{theorem}
\label{thm:CLT_double}Let $p>1,$ $\varepsilon>0,$ $\mu,$ $\nu\in
P_{2p+\varepsilon}\left(  \mathbb{R}^{d}\right)  $ satisfy one of the
uniqueness conditions of Lemma \ref{lm:uniq1} or Lemma \ref{lm:uniq2}. Let
$\varphi$ be the unique optimal potential for $D_{p}\left(  \mu,P_{\leqslant
\nu}\right)  $ with $m_{\nu}=\int yd\nu,$ $\varphi\left(  m_{\nu}\right)  =0$.
Assume $\frac{m}{m+n}\rightarrow\lambda\in\left(  0,1\right)  .$ Write
$\psi=Q_{p}\left(  \varphi\right)  .$ Then
\[
\sqrt{\frac{mn}{m+n}}\operatorname*{Var}\left(  R_{m,n}\right)  \rightarrow0
\]
and the central limit theorem holds%
\[
\sqrt{\frac{mn}{m+n}}\left[  T_{p}\left(  \mu_{m},P_{\leqslant\nu_{n}}\right)
-E\left(  T_{p}\left(  \mu_{m},P_{\leqslant\nu_{n}}\right)  \right)  \right]
\rightarrow\mathcal{N}\left(  0,\sigma_{p,\lambda}^{2}\right)  \text{
weakly},
\]
where (see $\left(  \ref{thm:CLT_eq2}\right)  $ for $\sigma_{p,\mu}^{2}$ and
$\left(  \ref{thm:CLT_vtx_eq0}\right)  $ for $\sigma_{p,\nu}^{2}$)
\[
\sigma_{p,\lambda}^{2}=\left(  1-\lambda\right)  \sigma_{p,\mu}^{2}%
+\lambda\sigma_{p,\nu}^{2}.
\]

\end{theorem}

\section{Shapeness of the moment assumption}

If $\mu,\nu$ additionally have bounded supports in Theorem \ref{thm:CLT_vtx},
then the optimal dual potentials are Lipschitz by Corollary
\ref{cor:dual_att_bd_spp}, hence the limit theorem remain valid even if
$\epsilon=0,$ and the moment assumption can be dropped completely: $\mu,$
$\nu\in P\left(  \mathbb{R}^{d}\right)  .$ But for generic measures, the
moment assumption of Theorem \ref{thm:CLT_vtx} (and Theorem
\ref{thm:CLT_double}) is probably sharp in the sense that the limit theorem
might not follow if we let $\epsilon=0.$ The difference between the
assumptions of Theorem \ref{thm:CLT} and Theorem \ref{thm:CLT_vtx} is on the
moments, with the latter requiring higher moments. But the only place we need
higher moment (i.e. $2p+\epsilon$ with $\epsilon>0$) in the proof of Theorem
\ref{thm:CLT_vtx} is when we are trying to show that $n|R_{n}-R_{n}^{\left(
j\right)  }|\rightarrow0$ in $L^{2}.$ With $\epsilon=0,$ the estimates
$\left(  \ref{var_bd_Snj_eq}\right)  $ gives no more information than that
$n^{2}E[S_{n}-S_{n}^{(j)}]^{2}$ is bounded$.$ Combining with the relation
$R_{n}=T_{p}\left(  \mu,P_{\leqslant\nu_{n}}\right)  +\int\varphi d\nu_{n},$
this shows that $n^{2}E[R_{n}-R_{n}^{\left(  j\right)  }]^{2}$ is bounded.
Meanwhile, note that the convergence $\left(  \ref{thm:CLT_vtx_eq1}\right)  $
that $n|R_{n}-R_{n}^{\left(  j\right)  }|\rightarrow0$ a.s. does not require
higher moments, so this part is free from the restriction of $\epsilon=0.$
Putting these all together we have
\begin{equation}
n[R_{n}-R_{n}^{\left(  j\right)  }]\rightarrow0\text{ a.s. and }%
n\operatorname*{Var}\left(  R_{n}\right)  \leqslant\frac{1}{2}n^{2}%
E[R_{n}-R_{n}^{\left(  j\right)  }]^{2}\leqslant A\text{ for some }A>0.
\label{esp0_bd}%
\end{equation}
This is not as strong as what we get when $\epsilon>0:$ $n|R_{n}%
-R_{n}^{\left(  j\right)  }|\rightarrow0$ in $L^{2}$. However, if we write
$V_{n}=n[R_{n}-R_{n}^{\left(  j\right)  }],$ then for every subsequence
$V_{n_{k}},$ $\left(  \ref{esp0_bd}\right)  $ allows us to extract a further
subsequence $V_{n_{k_{j}}}$ whose Ces\`{a}ro mean converges to $0$ in $L^{2}.$
Now using the hierarchical Efron-Stein inequality (Theorem \ref{thm:hier_ES})
we get%
\begin{equation}
\operatorname*{Var}\left(  \frac{1}{M}\sum_{j=1}^{M}\sqrt{j}R_{n_{k_{j}}%
}\right)  \leqslant\frac{1}{2}E\left(  \frac{1}{M}\sum_{j=1}^{M}V_{n_{k_{j}}%
}\right)  ^{2}\rightarrow0. \label{esp0_bd1}%
\end{equation}
So the sequence $\sqrt{n}(R_{n}-ER_{n})$ has the property that every
subsequence of it has a further subsequence whose Ces\`{a}ro mean converges to
$0$ in $L^{2}$ (hence converges to $0$ in distribution sense).

A fundamental theorem in real analysis says that a real sequence
$a_{n}\rightarrow a$ if and only if every subsequence has a further
subsequence whose Ces\`{a}ro mean converges to $a.$ If this holds for weak
convergence of random variables, then by using the fact that $\sqrt{n}%
[T_{p}\left(  \mu,P_{\leqslant\nu_{n}}\right)  -ET_{p}\left(  \mu
,P_{\leqslant\nu_{n}}\right)  ]$ is tight (due to $\left(
\ref{var_bd_Snj_eq1}\right)  $)$,$ we would infer from the equation%
\[
T_{p}\left(  \mu,P_{\leqslant\nu_{n}}\right)  =R_{n}-\int\varphi d\nu_{n}%
\]
that the weak limit of $\sqrt{n}[T_{p}\left(  \mu,P_{\leqslant\nu_{n}}\right)
-ET_{p}\left(  \mu,P_{\leqslant\nu_{n}}\right)  ]$ is $\mathcal{N}\left(
0,\sigma_{p,\nu}^{2}\right)  $ (see $\left(  \ref{thm:CLT_vtx_eq0}\right)  $
for $\sigma_{p,\nu}^{2}$). But this fundamental theorem in real analysis does
not extend to weak convergence in general, as shown by Example
\ref{eg:cesaro_m_dist} below. Indeed, for a general sequence $X_{n}$ of random
variables$,$ if every subsequence of $X_{n}$ has a further subsequence whose
Ces\`{a}ro mean converges weakly to the same limit $X,$ then we can show that
$X_{n}$ is tight. But this is much weaker than needed to identify the desired
weak limit of $\sqrt{n}[T_{p}\left(  \mu,P_{\leqslant\nu_{n}}\right)
-ET_{p}\left(  \mu,P_{\leqslant\nu_{n}}\right)  ]$.

\begin{example}
\label{eg:cesaro_m_dist}Let $Z_{n}\sim\mathcal{N}\left(  0,1\right)  $ be
i.i.d. Define%
\[
X_{n}=\frac{1}{\sqrt{n}}\sum_{k=1}^{n}Z_{k},\text{ }\forall n\geqslant1.
\]
The sequence $X_{n}\sim\mathcal{N}\left(  0,1\right)  $ is tight and $E\left(
X_{n}\right)  ^{2}=1,$ $n\geqslant1.$ $X_{n}$ trivially converges weakly to
$\mathcal{N}\left(  0,1\right)  .$ Define its Ces\`{a}ro mean%
\[
Y_{m}\triangleq\frac{1}{m}\sum_{n=1}^{m}X_{n}=\frac{1}{m}\sum_{n=1}^{m}\left(
\frac{1}{\sqrt{n}}\sum_{k=1}^{n}Z_{k}\right)  .
\]
Since $Y_{m}$ is a linear combination of independent normal variables, it must
be normal with a mean $0.$ To find its limit distribution, we just need to
find its variance as $m\rightarrow\infty.$ By swapping the order of summation,
we group by each $Z_{k}$%
\[
Y_{m}=\frac{1}{m}\sum_{k=1}^{m}Z_{k}\left(  \sum_{n=k}^{n}\frac{1}{\sqrt{n}%
}\right)  .
\]
Hence, using $\operatorname*{Var}\left(  Z_{k}\right)  =1,$%
\[
\operatorname*{Var}\left(  Y_{m}\right)  =\frac{1}{m^{2}}\sum_{k=1}%
^{m}\operatorname*{Var}\left(  Z_{k}\right)  \left(  \sum_{n=k}^{n}\frac
{1}{\sqrt{n}}\right)  ^{2}=\frac{1}{m}\sum_{k=1}^{m}\left(  \frac{1}{m}%
\sum_{n=k}^{n}\frac{1}{\sqrt{n/m}}\right)  ^{2}.
\]
The last term is recognized as a Riemann sum. Therefore%
\[
\lim_{m\rightarrow\infty}\operatorname*{Var}\left(  Y_{m}\right)  =\int%
_{0}^{1}\left(  \int_{x}^{1}\frac{1}{\sqrt{y}}dy\right)  ^{2}dx=\frac{2}{3}.
\]
This shows that%
\[
Y_{m}\rightarrow\mathcal{N}\left(  0,\frac{2}{3}\right)  \text{ weakly.}%
\]

\end{example}

Moreover, from the subsequence Ces\`{a}ro mean property $\left(
\ref{esp0_bd1}\right)  $ it does not generally follow that $\sqrt{n}%
(R_{n}-ER_{n})\rightarrow0$ weakly$\,,$ although it does follow that $\sqrt
{n}(R_{n}-ER_{n})$ is tight. The example below demonstrate this: there exists
a sequence $R_{n}$ such that $n\operatorname*{Var}\left(  R_{n}\right)  $ is
bounded, every subsequence of $\sqrt{n}(R_{n}-ER_{n})$ has a further
subsequence whose Ces\`{a}ro mean converges to $0$ in $L^{2}$, but $\sqrt
{n}(R_{n}-ER_{n})$ does not converge weakly to $0.$

\begin{example}
Let $\left(  \xi_{i}\right)  _{i\geqslant1}$ be i.i.d. Bernoulli variables
$P\left(  \xi_{i}=-1\right)  =P\left(  \xi_{i}=1\right)  =1/2$ and define%
\[
S_{n}=\sum_{i=1}^{n}\xi_{i},\text{ }R_{n}=\frac{S_{n}}{n}.
\]
Then $ER_{n}=0$, $R_{n}\rightarrow0$ a.s. and%
\[
n\operatorname*{Var}\left(  R_{n}\right)  =1.
\]
We will see that

$\left(  i\right)  $ Every subsequence of $Z_{n}\triangleq\sqrt{n}R_{n}$ has a
further subsequence whose Ces\`{a}ro mean converges to $0$ in $L^{2}$.

$\left(  ii\right)  $ However $Z_{n}\rightarrow N\left(  0,1\right)  $ weakly,
not $0.$
\end{example}

\begin{proof}
\textbf{1}. For $m\leqslant n$,%
\[
S_{n}=S_{m}+\sum_{i=m+1}^{n}\xi_{i}.
\]
The second sum is independent of $S_{m}$ and has mean $0$. Therefore%
\[
E\left(  Z_{m}Z_{n}\right)  =\frac{E\left(  S_{m}S_{n}\right)  }{\sqrt{mn}%
}=\frac{E\left(  S_{m}^{2}\right)  }{\sqrt{mn}}=\sqrt{\frac{m}{n}}.
\]

\textbf{2}. Consider an arbitrary subsequence $n_{1}<n_{2}<\cdots.$ Because
$n_{k}\rightarrow\infty$, we may choose a further subsequence%
\[
N_{j}=n_{k_{j}}%
\]
recursively so that%
\[
N_{j}\geqslant4^{j}N_{j-1},\ j\geqslant2.
\]
For $i<j$, $N_{i}\leqslant N_{j-1}.$ So%
\[
E\left(  Z_{N_{i}}Z_{N_{j}}\right)  =\sqrt{\frac{N_{i}}{N_{j}}}\leqslant
\sqrt{\frac{N_{j-1}}{N_{j}}}\leqslant2^{-j}.
\]
Define the Ces\`{a}ro mean%
\[
A_{M}=\frac{1}{M}\sum_{j=1}^{M}Z_{N_{j}}.
\]
Then%
\[
\left\Vert A_{M}\right\Vert _{L^{2}}^{2}=\frac{1}{M^{2}}\left[  \sum_{j=1}%
^{M}EZ_{N_{j}}^{2}+2\sum_{1\leqslant i<j\leqslant M}E\left(  Z_{N_{i}}%
Z_{N_{j}}\right)  \right]  \leqslant\frac{1}{M}+\frac{2}{M^{2}}\sum_{j=2}%
^{M}\left(  j-1\right)  2^{-j}.
\]
Since the sum in the rightmost converges as $M\rightarrow\infty,$ we obtain%
\[
\left\Vert A_{M}\right\Vert _{L^{2}}\rightarrow0\text{ as }M\rightarrow
\infty.
\]

\textbf{3}. The characteristic function of $Z_{n}=S_{n}/\sqrt{n}$ is%
\[
\varphi_{n}\left(  t\right)  =E\exp\left(  it\frac{S_{n}}{\sqrt{n}}\right)
=\prod_{k=1}^{n}E\exp\left(  it\frac{\xi_{k}}{\sqrt{n}}\right)  =\left(
\cos\frac{t}{\sqrt{n}}\right)  ^{n}.
\]
Therefore, $\forall t,$%
\[
\varphi_{n}\left(  t\right)  \rightarrow e^{-t^{2}/2}\text{ as }%
n\rightarrow\infty.
\]
By L\'{e}vy's continuity theorem,%
\[
Z_{n}\rightarrow N\left(  0,1\right)  \text{ weakly, not }0.
\]

\end{proof}

\section{Appendix: A hierarchical Efron-Stein inequality}

Let $X_{1},...,X_{n}$ be independent and $Z$ is a hierarchical function of
$X_{1},...,X_{n}$ of the form%
\[
Z=Z\left(  X_{1},...,X_{n}\right)  =U_{1}\left(  X_{1}\right)  +U_{2}\left(
X_{1},X_{2}\right)  +\cdot\cdot\cdot+U_{n}\left(  X_{1},...,X_{n}\right)  .
\]
Also let $X_{1}^{\prime},...,X_{n}^{\prime}$ be independent copies of
$X_{1},...,X_{n}$. For $i\geqslant j,$ write
\[
U_{ij}^{\prime}=U_{i}\left(  X_{1},...,X_{j}^{\prime},...X_{n}\right)  .
\]
When $j=1,$ we simply write $U_{i}^{\prime}=U_{i}\left(  X_{1}^{\prime}%
,X_{2},...,X_{n}\right)  .$ For $i<j,$ we define $U_{ij}^{\prime}=U_{i}$.

\begin{theorem}
\label{thm:hier_ES}It holds that%
\[
\operatorname*{Var}\left(  Z\right)  \leqslant\frac{1}{2}E\left(  \sum
_{i=1}^{n}\sqrt{i}\left\vert U_{i}-U_{i}^{\prime}\right\vert \right)  ^{2}.
\]

\end{theorem}

\begin{proof}
The proof is built upon the original proof of Efron-Stein inequality
\cite[Theorem 3.1]{boucheron2003concentration}. Define $E_{0}\left(
\cdot\right)  $ as the expectation $E\left(  \cdot\right)  ,$ and for
$i\geqslant1,$%
\[
E_{i}\left(  \cdot\right)  \triangleq E\left(  \cdot\mid X_{1},X_{2}%
,...,X_{i}\right)  ,\text{ }E^{\left(  i\right)  }\left(  \cdot\right)
\triangleq E\left(  \cdot\mid X_{1},X_{2},...,X_{i-1},X_{i},...,X_{n}\right)
.
\]
For each $i=1,...,n$, the action of $U_{i}$ along the $j$-th variable
($j\leqslant i$) is defined as%
\[
\Delta_{ij}=E_{j}\left(  U_{i}-E^{\left(  j\right)  }\left(  U_{i}\right)
\right)  =E_{j}U_{i}-E_{j-1}U_{i}.
\]
The definition still makes sense when $j>i$, in this case $\Delta_{ij}=0$.
Decompose $U_{i}$ as sum of actions along individual variables,%
\[
U_{i}-EU_{i}=\sum_{j=1}^{i}\Delta_{ij}=\sum_{j=1}^{i}E_{j}\left(
U_{i}-E^{\left(  j\right)  }\left(  U_{i}\right)  \right)  .
\]
So%
\[
Z-EZ=\sum_{i=1}^{n}U_{i}=\sum_{i=1}^{n}\sum_{j=1}^{i}\Delta_{ij}=\sum
_{j=1}^{n}\left(  \sum_{i=j}^{n}\Delta_{ij}\right)  .
\]
The sum in the brackets gives the total action of $Z$ along the $j$-th
variable
\[
\sum_{i=j}^{n}\Delta_{ij}=\sum_{i=j}^{n}E_{j}\left(  U_{i}-E^{\left(
j\right)  }\left(  U_{i}\right)  \right)  =E_{j}\left[  \sum_{i=j}^{n}%
U_{i}-E^{\left(  j\right)  }\left(  \sum_{i=j}^{n}U_{i}\right)  \right]  .
\]
Note $E\Delta_{ij}\Delta_{kl}=0$ whenever $j\neq l$. Therefore the total
variance is the sum of coordinate-wise variances, and using Jensen inequality
gives%
\[
\operatorname*{Var}\left(  Z\right)  =\sum_{j=1}^{n}E\left(  \sum_{i=j}%
^{n}\Delta_{ij}\right)  ^{2}\leqslant\sum_{j=1}^{n}EE_{j}\left[  \sum
_{i=j}^{n}U_{i}-E^{\left(  j\right)  }\left(  \sum_{i=j}^{n}U_{i}\right)
\right]  ^{2}=\sum_{j=1}^{n}E\left[  \left.  \operatorname*{Var}\right.
^{\left(  j\right)  }\left(  \sum_{i=j}^{n}U_{i}\right)  \right]  .
\]
By the classical representation of variance, the equation continues%
\[
=\frac{1}{2}\sum_{j=1}^{n}E\left(  \sum_{i=j}^{n}U_{i}-\sum_{i=j}^{n}%
U_{ij}^{\prime}\right)  ^{2}=\frac{1}{2}\sum_{j=1}^{n}E\left(  \sum_{i=1}%
^{n}U_{i}-\sum_{i=1}^{n}U_{ij}^{\prime}\right)  ^{2}\text{ (since }%
U_{i}=U_{ij}^{\prime}\text{ for }i<j\text{)}%
\]
now developing the square%
\begin{align*}
&  =\frac{1}{2}\sum_{j=1}^{n}E\left(  \sum_{i=1}^{n}\left(  U_{i}%
-U_{ij}^{\prime}\right)  \right)  ^{2}=\frac{1}{2}\sum_{j=1}^{n}E\left(
\sum_{i,k}\left(  U_{i}-U_{ij}^{\prime}\right)  \left(  U_{k}-U_{kj}^{\prime
}\right)  \right) \\
&  \leqslant\frac{1}{2}\sum_{j=1}^{n}E\left(  \sum_{i,k}\left\vert
U_{i}-U_{ij}^{\prime}\right\vert \left\vert U_{k}-U_{kj}^{\prime}\right\vert
\right)  =\frac{1}{2}E\left(  \sum_{i,k}\sum_{j=1}^{n}\left\vert U_{i}%
-U_{ij}^{\prime}\right\vert \left\vert U_{k}-U_{kj}^{\prime}\right\vert
\right)
\end{align*}
the term in the sum is non-zero only if $j\leqslant\min\left(  i,k\right)  ,$
also since $\left\vert U_{i}-U_{ij}^{\prime}\right\vert \left\vert
U_{k}-U_{kj}^{\prime}\right\vert $ has the same distribution as $\left\vert
U_{i}-U_{i}^{\prime}\right\vert \left\vert U_{k}-U_{k}^{\prime}\right\vert ,$%
\begin{align*}
&  =\frac{1}{2}E\left(  \sum_{i,s}\sum_{j\leqslant\min\left(  i,k\right)
}\left\vert U_{i}-U_{ij}^{\prime}\right\vert \left\vert U_{k}-U_{kj}^{\prime
}\right\vert \right)  =\frac{1}{2}E\left(  \sum_{i,k}\sum_{j\leqslant
\min\left(  i,k\right)  }\left\vert U_{i}-U_{i}^{\prime}\right\vert \left\vert
U_{k}-U_{k}^{\prime}\right\vert \right) \\
&  \leqslant\frac{1}{2}E\left(  \sum_{i,s}\sqrt{i}\sqrt{k}\left\vert
U_{i}-U_{i}^{\prime}\right\vert \left\vert U_{k}-U_{k}^{\prime}\right\vert
\right)  \text{ (since }\min\left(  i,k\right)  \leqslant\sqrt{i}\sqrt
{k}\text{)}\\
&  =\frac{1}{2}E\left(  \sum_{i}\sqrt{i}\left\vert U_{i}-U_{i}^{\prime
}\right\vert \right)  ^{2}.
\end{align*}

\end{proof}

\section*{Acknowledgement}
The author sincerely thanks Young-heon Kim for the hospitality and inspiration, the gratitude extends to Jakwang Kim and Andrew Warren for valuable exchange of ideas while the paper was in a prototype stage. This project would not have crossed the finish line without their encouragement.

\bibliographystyle{alpha}
\bibliography{main}

\newcommand{\etalchar}[1]{$^{#1}$}
\begin{thebibliography}{dBGSLRV25}

\bibitem[ACJ{\etalchar{+}}20]{alfonsi2020sampling}
Aur{\'e}lien Alfonsi, Jacopo Corbetta, Benjamin Jourdain, et~al.
\newblock Sampling of probability measures in the convex order by wasserstein
  projection.
\newblock {\em Annales de l'Institut Henri Poincar{\'e}, Probabilit{\'e}s et
  Statistiques}, 56(3):1706--1729, 2020.

\bibitem[AGS05]{ambrosio2005gradient}
Luigi Ambrosio, Nicola Gigli, and Giuseppe Savar{\'e}.
\newblock {\em Gradient flows: in metric spaces and in the space of probability
  measures}.
\newblock Springer, 2005.

\bibitem[AJ26]{alfonsi2026wasserstein}
Aur{\'e}lien Alfonsi and Benjamin Jourdain.
\newblock Wasserstein projections in the convex order: regularity and
  characterization in the quadratic gaussian case.
\newblock {\em Electronic Journal of Probability}, 31:1--29, 2026.

\bibitem[BGMN26]{bourne2026semi}
David~P Bourne, Thomas Gallou{\~A}{\c{G}}t, Quentin M{\~A}{\v{S}}rigot, and
  Andrea Natale.
\newblock Semi-discrete convex order and laguerre tessellation fitting.
\newblock {\em arXiv preprint arXiv:2606.29913}, 2026.

\bibitem[BLB03]{boucheron2003concentration}
St{\'e}phane Boucheron, G{\'a}bor Lugosi, and Olivier Bousquet.
\newblock Concentration inequalities.
\newblock In {\em Summer school on machine learning}, pages 208--240. Springer,
  2003.

\bibitem[BLM03]{boucheron2013concentration}
St{\'e}phane Boucheron, G{\'a}bor Lugosi, and Pascal Massart.
\newblock {\em Concentration Inequalities: A Nonasymptotic Theory of
  Independence,(2013)}.
\newblock OUP: Oxford, 2003.

\bibitem[dBGSL24]{del2024central}
Eustasio del Barrio, Alberto Gonz{\'a}lez~Sanz, and Jean-Michel Loubes.
\newblock Central limit theorems for semi-discrete wasserstein distances.
\newblock {\em Bernoulli}, 30(1):554--580, 2024.

\bibitem[dBGSLRV25]{del2025distributional}
Eustasio del Barrio, Alberto Gonz{\'a}lez-Sanz, Jean-Michel Loubes, and David
  Rodr{\'\i}guez-V{\'\i}tores.
\newblock Distributional limit theory for optimal transport.
\newblock {\em arXiv preprint arXiv:2505.19104}, 2025.

\bibitem[DBL19]{del2019central}
Eustasio Del~Barrio and Jean-Michel Loubes.
\newblock Central limit theorems for empirical transportation cost in general
  dimension.
\newblock {\em The Annals of Probability}, 47(2):926--951, 2019.

\bibitem[dBSLNW23]{del2023improved}
Eustasio del Barrio, Alberto~Gonz{\'a}lez Sanz, Jean-Michel Loubes, and
  Jonathan Niles-Weed.
\newblock An improved central limit theorem and fast convergence rates for
  entropic transportation costs.
\newblock {\em SIAM Journal on Mathematics of Data Science}, 5(3):639--669,
  2023.

\bibitem[ES81]{efron1981jackknife}
Bradley Efron and Charles Stein.
\newblock The jackknife estimate of variance.
\newblock {\em The Annals of Statistics}, pages 586--596, 1981.

\bibitem[Eva10]{evans2010partial}
Lawrence~C Evans.
\newblock {\em Partial differential equations}, volume~19.
\newblock American mathematical society, 2010.

\bibitem[For26]{ford2026quantitative}
William Ford.
\newblock Quantitative uniqueness of kantorovich potentials.
\newblock {\em arXiv preprint arXiv:2603.29595}, 2026.

\bibitem[GJ20]{gozlan2020mixture}
Nathael Gozlan and Nicolas Juillet.
\newblock On a mixture of brenier and strassen theorems.
\newblock {\em Proceedings of the London Mathematical Society},
  120(3):434--463, 2020.

\bibitem[GKRS24]{goldfeld2024limit}
Ziv Goldfeld, Kengo Kato, Gabriel Rioux, and Ritwik Sadhu.
\newblock Limit theorems for entropic optimal transport maps and sinkhorn
  divergence.
\newblock {\em Electronic Journal of Statistics}, 18(1):980--1041, 2024.

\bibitem[GRST17]{gozlan2017kantorovich}
Nathael Gozlan, Cyril Roberto, Paul-Marie Samson, and Prasad Tetali.
\newblock Kantorovich duality for general transport costs and applications.
\newblock {\em Journal of Functional Analysis}, 273(11):3327--3405, 2017.

\bibitem[GSLNW22]{gonzalez2022weak}
Alberto Gonzalez-Sanz, Jean-Michel Loubes, and Jonathan Niles-Weed.
\newblock Weak limits of entropy regularized optimal transport; potentials,
  plans and divergences.
\newblock {\em arXiv preprint arXiv:2207.07427}, 2022.

\bibitem[Jou25]{fields2025Benjamin}
Benjamin Jourdain.
\newblock Wasserstein projections in the convex order.
\newblock In {\em Optimal transport: stochastics, projections, and
  applications}. The Fields Institute for Research in Mathematical Sciences,
  2025.

\bibitem[KKN26]{kim2025stability}
Jakwang Kim, Young-Heon Kim, and Andrea Natale.
\newblock Stability of wasserstein projections in convex order via metric
  extrapolation.
\newblock {\em Electronic Communications in Probability}, 31:1 -- 12, 2026.

\bibitem[KKRW]{kim2024statistical}
Jakwang Kim, Young-Heon Kim, Yuanlong Ruan, and Andrew Warren.
\newblock Statistical inference of convex order by wasserstein projection.
\newblock {\em Bernoulli (to appear)}.

\bibitem[KR24]{yh_yl_stochastic_order}
Young-Heon Kim and Yuanlong Ruan.
\newblock Backward and forward wasserstein projections in stochastic order.
\newblock {\em Journal of Functional Analysis}, 286(2):110201, 2024.

\bibitem[SHM25]{staudt2025uniqueness}
Thomas Staudt, Shayan Hundrieser, and Axel Munk.
\newblock On the uniqueness of kantorovich potentials.
\newblock {\em SIAM Journal on Mathematical Analysis}, 57(2):1452--1482, 2025.

\end{thebibliography}
\end{document}